\UseRawInputEncoding
\documentclass[12pt]{amsart}
\usepackage{amssymb}
\usepackage[all]{xy}
\usepackage[toc,page,title,titletoc,header]{appendix}
\usepackage{xcolor}

\begin{document}
\theoremstyle{plain}
\newtheorem{thm}{Theorem}[section]
\newtheorem*{thm1}{Theorem 1}
\newtheorem*{thm1.1}{Theorem 1.1}
\newtheorem*{thmM}{Main Theorem}
\newtheorem*{thmA}{Theorem A}
\newtheorem*{thm2}{Theorem 2}
\newtheorem{lemma}[thm]{Lemma}
\newtheorem{lem}[thm]{Lemma}
\newtheorem{cor}[thm]{Corollary}
\newtheorem{pro}[thm]{Proposition}
\newtheorem{propose}[thm]{Proposition}
\newtheorem{variant}[thm]{Variant}
\theoremstyle{definition}
\newtheorem{notations}[thm]{Notations}
\newtheorem{rem}[thm]{Remark}
\newtheorem{rmk}[thm]{Remark}
\newtheorem{rmks}[thm]{Remarks}
\newtheorem{defi}[thm]{Definition}
\newtheorem{exe}[thm]{Example}
\newtheorem{claim}[thm]{Claim}
\newtheorem{ass}[thm]{Assumption}
\newtheorem{prodefi}[thm]{Proposition-Definition}
\newtheorem{que}[thm]{Question}
\newtheorem{con}[thm]{Conjecture}
\newtheorem*{assa}{Assumption A}
\newtheorem*{algstate}{Algebraic form of Theorem \ref{thmstattrainv}}

\newtheorem*{dmlcon}{Dynamical Mordell-Lang Conjecture}
\newtheorem*{zdocon}{Zariski dense orbit Conjecture}
\newtheorem*{pdd}{P(d)}
\newtheorem*{bfd}{BF(d)}

\newtheorem*{probreal}{Realization problems}
\numberwithin{equation}{section}
\newcounter{elno}                
\def\points{\list
{\hss\llap{\upshape{(\roman{elno})}}}{\usecounter{elno}}}
\let\endpoints=\endlist
\newcommand{\SH}{\rm SH}
\newcommand{\Tan}{\rm Tan}
\newcommand{\res}{\rm res}
\newcommand{\Om}{\Omega}
\newcommand{\om}{\omega}
\newcommand{\la}{\lambda}
\newcommand{\mc}{\mathcal}
\newcommand{\mb}{\mathbb}
\newcommand{\surj}{\twoheadrightarrow}
\newcommand{\inj}{\hookrightarrow}
\newcommand{\zar}{{\rm zar}}
\newcommand{\Exc}{{\rm Exc}}
\newcommand{\an}{{\rm an}}
\newcommand{\red}{{\rm red}}
\newcommand{\codim}{{\rm codim}}
\newcommand{\Supp}{{\rm Supp}}
\newcommand{\rank}{{\rm rank}}
\newcommand{\Ker}{{\rm Ker \ }}
\newcommand{\Pic}{{\rm Pic}}
\newcommand{\Der}{{\rm Der}}
\newcommand{\Div}{{\rm Div}}
\newcommand{\Hom}{{\rm Hom}}
\newcommand{\im}{{\rm im}}
\newcommand{\Spec}{{\rm Spec \,}}
\newcommand{\Nef}{{\rm Nef \,}}
\newcommand{\Frac}{{\rm Frac \,}}
\newcommand{\Sing}{{\rm Sing}}
\newcommand{\sing}{{\rm sing}}
\newcommand{\reg}{{\rm reg}}
\newcommand{\Char}{{\rm char}}
\newcommand{\Tr}{{\rm Tr}}
\newcommand{\ord}{{\rm ord}}
\newcommand{\id}{{\rm id}}
\newcommand{\NE}{{\rm NE}}
\newcommand{\Gal}{{\rm Gal}}
\newcommand{\Min}{{\rm Min \ }}
\newcommand{\Max}{{\rm Max \ }}
\newcommand{\Alb}{{\rm Alb}\,}
\newcommand{\GL}{{\rm GL}\,}        
\newcommand{\PGL}{{\rm PGL}\,}
\newcommand{\Bir}{{\rm Bir}}
\newcommand{\Aut}{{\rm Aut}}
\newcommand{\End}{{\rm End}}
\newcommand{\Per}{{\rm Per}\,}
\newcommand{\ie}{{\it i.e.\/},\ }
\newcommand{\niso}{\not\cong}
\newcommand{\nin}{\not\in}
\newcommand{\soplus}[1]{\stackrel{#1}{\oplus}}
\newcommand{\by}[1]{\stackrel{#1}{\rightarrow}}
\newcommand{\longby}[1]{\stackrel{#1}{\longrightarrow}}
\newcommand{\vlongby}[1]{\stackrel{#1}{\mbox{\large{$\longrightarrow$}}}}
\newcommand{\ldownarrow}{\mbox{\Large{\Large{$\downarrow$}}}}
\newcommand{\lsearrow}{\mbox{\Large{$\searrow$}}}
\renewcommand{\d}{\stackrel{\mbox{\scriptsize{$\bullet$}}}{}}
\newcommand{\dlog}{{\rm dlog}\,}    
\newcommand{\longto}{\longrightarrow}
\newcommand{\vlongto}{\mbox{{\Large{$\longto$}}}}
\newcommand{\limdir}[1]{{\displaystyle{\mathop{\rm lim}_{\buildrel\longrightarrow\over{#1}}}}\,}
\newcommand{\liminv}[1]{{\displaystyle{\mathop{\rm lim}_{\buildrel\longleftarrow\over{#1}}}}\,}
\newcommand{\norm}[1]{\mbox{$\parallel{#1}\parallel$}}
\newcommand{\boxtensor}{{\Box\kern-9.03pt\raise1.42pt\hbox{$\times$}}}
\newcommand{\into}{\hookrightarrow}
\newcommand{\image}{{\rm image}\,}
\newcommand{\Lie}{{\rm Lie}\,}      
\newcommand{\CM}{\rm CM}
\newcommand{\sext}{\mbox{${\mathcal E}xt\,$}}  
\newcommand{\shom}{\mbox{${\mathcal H}om\,$}}  
\newcommand{\coker}{{\rm coker}\,}  
\newcommand{\sm}{{\rm sm}}
\newcommand{\pgcd}{\text{pgcd}}
\newcommand{\trd}{\text{tr.d.}}
\newcommand{\tensor}{\otimes}
\newcommand{\hotimes}{\hat{\otimes}}

\renewcommand{\iff}{\mbox{ $\Longleftrightarrow$ }}
\newcommand{\supp}{{\rm supp}\,}
\newcommand{\ext}[1]{\stackrel{#1}{\wedge}}
\newcommand{\onto}{\mbox{$\,\>>>\hspace{-.5cm}\to\hspace{.15cm}$}}
\newcommand{\propsubset}
{\mbox{$\textstyle{
\subseteq_{\kern-5pt\raise-1pt\hbox{\mbox{\tiny{$/$}}}}}$}}
\newcommand{\sA}{{\mathcal A}}
\newcommand{\sB}{{\mathcal B}}
\newcommand{\sC}{{\mathcal C}}
\newcommand{\sD}{{\mathcal D}}
\newcommand{\sE}{{\mathcal E}}
\newcommand{\sF}{{\mathcal F}}
\newcommand{\sG}{{\mathcal G}}
\newcommand{\sH}{{\mathcal H}}
\newcommand{\sI}{{\mathcal I}}
\newcommand{\sJ}{{\mathcal J}}
\newcommand{\sK}{{\mathcal K}}
\newcommand{\sL}{{\mathcal L}}
\newcommand{\sM}{{\mathcal M}}
\newcommand{\sN}{{\mathcal N}}
\newcommand{\sO}{{\mathcal O}}
\newcommand{\sP}{{\mathcal P}}
\newcommand{\sQ}{{\mathcal Q}}
\newcommand{\sR}{{\mathcal R}}
\newcommand{\sS}{{\mathcal S}}
\newcommand{\sT}{{\mathcal T}}
\newcommand{\sU}{{\mathcal U}}
\newcommand{\sV}{{\mathcal V}}
\newcommand{\sW}{{\mathcal W}}
\newcommand{\sX}{{\mathcal X}}
\newcommand{\sY}{{\mathcal Y}}
\newcommand{\sZ}{{\mathcal Z}}
\newcommand{\A}{{\mathbb A}}
\newcommand{\B}{{\mathbb B}}
\newcommand{\C}{{\mathbb C}}
\newcommand{\D}{{\mathbb D}}
\newcommand{\E}{{\mathbb E}}
\newcommand{\F}{{\mathbb F}}
\newcommand{\G}{{\mathbb G}}
\newcommand{\HH}{{\mathbb H}}
\newcommand{\I}{{\mathbb I}}
\newcommand{\J}{{\mathbb J}}
\newcommand{\M}{{\mathbb M}}
\newcommand{\N}{{\mathbb N}}
\renewcommand{\P}{{\mathbb P}}
\newcommand{\Q}{{\mathbb Q}}
\newcommand{\R}{{\mathbb R}}
\newcommand{\T}{{\mathbb T}}
\newcommand{\U}{{\mathbb U}}
\newcommand{\V}{{\mathbb V}}
\newcommand{\W}{{\mathbb W}}
\newcommand{\X}{{\mathbb X}}
\newcommand{\Y}{{\mathbb Y}}
\newcommand{\Z}{{\mathbb Z}}
\newcommand{\bk}{{\mathbf{k}}}
\newcommand{\bbk}{{\overline{\mathbf{k}}}}
\newcommand{\Fix}{\mathrm{Fix}}

\title[]{Remarks on algebraic dynamics in positive characteristic}

\author{Junyi Xie}


\address{Univ Rennes, CNRS, IRMAR - UMR 6625, F-35000 Rennes, France}

\email{junyi.xie@univ-rennes1.fr}

\thanks{The author is partially supported by project ``Fatou'' ANR-17-CE40-0002-01 and  PEPS CNRS}

\date{\today}

\bibliographystyle{plain}

\maketitle

\begin{abstract}
In this paper, we study arithmetic dynamics in arbitrary characteristic, in particular in positive characteristic.
We generalise some basic facts on arithmetic degree and canonical height in positive characteristic. 
As applications, we prove the dynamical Mordell-Lang conjecture for automorphisms of projective surfaces of positive entropy, the Zariski dense orbit conjecture for automorphisms of projective surfaces and for endomorphisms of projective varieties with large first dynamical degree. We also study ergodic theory for constructible topology. For example, we prove the equidistribution of backward orbits for finite flat endomorphisms with large topological degree. As applications, we give a simple proof for weak dynamical Mordell-Lang and prove a counting result for backward orbits without multiplicities. This gives some applications for equidistributions on Berkovich spaces.
\end{abstract}


\section{Introduction}
Let $\bk$ be an algebraically closed field.  In this paper, most of the time (from Section 2 to Section 4),  we are mainly interested in the case $\Char\, \bk>0$.

Many problems in arithmetic dynamics, such as Dynamical Mordell-Lang conjecture, Zariski dense orbit conjecture are proposed in characteristic $0$. Indeed, their original statements do not hold in positive characteristic. But their known counter-examples often involve some Frobenius actions or some group structures. 
We suspect that the original statement of these conjecture still valid for ``general" dynamical systems in positive characteristic.  

\medskip

The $p$-adic interpolation lemma (\cite[Theorem 1]{Poonen2014} and \cite[Theorem 3.3]{Bell2010}) is a fundamental tool in arithmetic dynamics.
It has important applications in Dynamical Mordell-Lang and Zariski dense orbit conjecture \cite{Bell2016, Bell2010, Amerik2008, E.Amerik2011, Xie2019}.
But this lemma does not work in positive characteristic. Because this, some very basic cases of Dynamical Mordell-Lang and Zariski dense orbit conjecture are still open in positive characteristic. We hope that some corollaries of the $p$-adic interpolation lemma  still survive in positive characteristic. For this, I propose the following conjecture.
\begin{con}Set $K:=\overline{\F_p}((t))$ and $K^{\circ}=\overline{\F_p}[[t]]$ its valuation ring.
Let $f: (K^{\circ})^r\to (K^{\circ})^r$ be an analytic automorphism satisfying 
$f=\id \mod t.$ If there is no $n\geq 1$ such that $f^n=\id$, then the $f$-periodic points are not dense in $(K^{\circ})^r$ w.r.t.  $t$-adic topology.
\end{con}
On the other hand, we observed that, under certain assumption on the complexity of $f$, a global argument using height can be used to replace the local argument using the $p$-adic interpolation lemma.
We generalise the notion of arithmetic degree and prove some basic properties of it in positive characteristic. In particular, we generalise Kawaguchi-Silverman-Matsuzawa's upper bound for arithmetic degree \cite[Theorem 1.4]{Matsuzawa2020a} in positive characteristic. With such notion, we apply our observation to dynamical system in positive characteristic. In particular, we prove the Dynamical Mordell-Lang and Zariski dense orbit conjecture in some cases (see Section \ref{subintrodml} and \ref{subintrozdo}).

\medskip

Another aim of this paper is to study the ergodic theory on algebraic variety w.r.t constructible topology. Using this, we get some equidistribution reults and apply them to get some 
weak verisons of Dynamical Mordell-Lang, Manin-Mumford conjecture in arbitrary characteristic.   This also gives some applications for equidistributions on Berkovich spaces.
\subsection{Dynamical Mordell-Lang conjecture}\label{subintrodml}
Let $X$ be a  variety over $\bk$ and $f: X\dashrightarrow X$ be a rational self-map.

\begin{defi} We say $(X,f)$ satisfies the \emph{DML} property if for every $x\in X(\bk)$ whose $f$-orbit is well defined and every subvariety $V$ of $X$, the set $\{n\geq 0|\,\, f^n(x)\in V\}$ is a finite union of arithmetic progressions. 
\end{defi}
Here an arithmetic progression is a set of the form $\{an + b|\,\, n\in \N\}$ with $a,b \in \N$ possibly with $a = 0$.

\begin{dmlcon}If $\Char\,\bk=0$, then $(X,f)$ satisfies the DML property.
\end{dmlcon}

It was proved when $f$ is unramified \cite{Bell2010} and when $f$ is an endomorphism of $\A^2_{\overline{\Q}}$ \cite{Xie2017a}.
See \cite{Bell2016, Ghioca2018} for other known results.
In general, this conjecture does not hold in positive characteristic. An example is \cite[Example 3.4.5.1]{Bell2016} as follows (see \cite{Ghioca2019a,Corvaja2021} for more examples).
\begin{exe}\label{exenotdml} Let $\bk=\overline{\F_p(t)}$, $f: \A^2\to \A^2$ be the endomorphism defined by $(x,y)\mapsto (tx, (1-t)y).$
Set $V:=\{x-y=0\}$ and $e=(1,1).$ Then $\{n\geq 0|\,\, f^n(e)\in V\}=\{p^n|\,\, n\geq 0\}.$
\end{exe}

In \cite[Conjecture 13.2.0.1]{Bell2016}, Ghioca  and Scanlon proposed a variant of the Dynamical Mordell-Lang conjecture in positive characteristic (=$p$-DML), which asked $\{n\geq 0|\,\, f^n(x)\in V\}$ to be a finite union of arithmetic progressions along with finitely many 
sets taking form $$\{\sum_{i=1}^mc_ip^{l_in_i}|\,\, n_i\in \Z_{\geq 0}, i=1,\dots,m\}$$
where $m\in \Z_{>1}, k_i\in \Z_{\geq 0}, c_i\in \Q.$ See \cite{Ghioca2019a,Corvaja2021} for known results of $p$-DML. However, we suspect that for a ``general" dynamical system in positive characteristic still has the DML property.

\begin{thm}\label{dmlsurface}Let $X$ be a projective surface over $\bk.$ Let $f: X\to X$ be an automorphism. Assume that $\la_1(f)>1$. Then the pair $(X,f)$ satisfies the DML property.  
\end{thm}
Here $\la_i(f)$ is the $i$-th dynamical degree of $f$ (see Section \ref{subsectiondydeg}). 
The following is a similar result for birational endomorphisms of $\A^2.$ In \cite[Theorem A]{Xie2014}, it is stated in characteristic $0$. But when $\la_1(f)>1$, its proof works in any characteristic. 
\begin{thm}\cite[Theorem A]{Xie2014}\label{thmdmlat} Let $f:\A^2\to \A^2$ be a birational endomorphism over $\bk$. If $\la_1(f)>1$, $(\A^2,f)$ satisfies the DML property.
\end{thm}

\subsection{Zariski dense orbit conjecture}\label{subintrozdo}
Let $X$ be a variety over $\bk$ and $f: X\dashrightarrow X$ be a dominant rational self-map. 
Denote by $\bk(X)^f$ the field of $f$-invariant rational functions on $X.$
Let $X_f(\bk)$ is the set of $X(\bk)$ whose orbit is well-defined.
For $x\in X_f(\bk)$, $O_f(x)$ is the orbit of $x.$

\begin{defi}
We say $(X,f)$ satisfies the \emph{SZDO} property if  there is $x\in X_f(\bk)$ such that $O_f(x)$ is Zariski dense in $X.$

We say $(X,f)$ satisfies the \emph{ZDO} property if 
either $\bk(X)^f\neq \bk$ or it satisfies SZDO property.
\end{defi}

The Zariski dense orbit conjecture was proposed by Medvedev and Scanlon \cite[Conjecture 5.10]{Medvdevv1}, by Amerik, Bogomolov and Rovinsky \cite{E.Amerik2011} and strengthens a conjecture of Zhang \cite{zhang}.
\begin{zdocon}\label{conexistszdo}If $\Char\,\bk=0$, then $(X,f)$ satisfies the ZDO property.
\end{zdocon}

This conjecture was proved for endomorphisms of projective surfaces \cite{Jia2020, Xie2019}, 
endomorphisms of $(\P^1)^N$ \cite{Medvdev,Xie2019} and endomorphisms of $\A^2$ \cite{Xie2017}.
See \cite{Amerik2008,Amerik,E.Amerik2011,Fakhruddin2014,Bell2017,Bell2017a,Ghioca2017a,Ghioca2018b,Ghioca2019, Bell, Jia2021}
for other known results. 

\medskip

The original statement of Zariski dense orbit conjecture is not true in characteristic $p>0$. It is completely wrong over $\bk=\overline{\F_p}$ and has counter-examples even when 
$\trd_{\overline{F_p}} \bk\geq 1$ (see \cite[Section 1.6]{Xie2019} and \cite[Remark 1.2]{Ghioca}).
Concerning the variants of the Zariski dense orbit conjecture in positive characteristic proposed in \cite[Section 1.6]{Xie2019} and \cite[Conjecture 1.3]{Ghioca}, 
we get the following result.
\begin{pro}\label{protautzdo}Let $K$ be an algebraically closed field extension of $\bk$ with $\trd_{\bk}K\geq \dim X$.
Then $(f_K,X_K)$ satisfies the ZDO property. Here $X_{K}$ and $f_K$ are the base change by $K$ of $X$ and $f.$
\end{pro}

The following example shows that the assumption $\trd_{\bk}K\geq \dim X$ is sharp. 
\begin{exe}Let $X$ be a variety over $\bk:=\overline{\F_p}$ of dimension $d\geq 1$. Assume that $X$ is defined over $\F_p$. Let $F: X\to X$ be the Frobenius endomorphism. 
It is clear that $\overline{\F_p}(X)^F=\overline{\F_p}$. For every algebraically closed field extension $K$ of $\bk$ with $\trd_{\bk}K\leq d-1$, and every $x\in X_K(K)$, $O_{F_K}(x)$ is not Zariski dense in $X_K.$
\end{exe}

On the other hand, the known counter-examples often involve some Frobenius actions. 
See \cite[Theorem 1.5, Question 1.7]{Ghioca} for  this phenomenon.
We suspect that when $\trd_{\overline{\F_p}} \bk\geq 1,$
a ``general" dynamical system in positive characteristic still have the ZDO property.
Applying arguments using height, we get the following results.

\begin{thm}\label{thmendononpre}Assume that $\Char\, \bk=p>0$ and $\trd_{\overline{F_p}} \bk\geq 1.$ 
Let $f: X\to X$ be a dominant endomorphism of a projective variety. If $\la_1(f)>1$, then for every nonempty Zariski open subset $U$ of $X$, there is $x\in U(\bk)$ with infinite orbit and $O_f(x)\subseteq U$.
\end{thm}
Theorem \ref{thmendononpre} can be viewed as a weak version of \cite[Corollary 9]{Amerik} in positive characteristic.

\begin{thm}\label{thmzdosuraut}Assume that $\Char\, \bk=p>0$ and $\trd_{\overline{F_p}} \bk\geq 1.$ 
Let $f: X\to X$ be an automorphism of a projective surface. Then $(X,f)$ satisfies the ZDO property.
\end{thm}

The following result is a generalization of \cite[Theorem 1.12 (iii)]{Jia2021} in positive characteristic. 
\begin{thm}\label{thmzdola}Assume that $\Char\, \bk=p>0$ and $\trd_{\overline{F_p}} \bk\geq 1.$ 
Let $f: X\to X$ be a dominant endomorphism of a projective variety.
Assume that  $X$ is smooth of dimension $d\geq 2$, and $\la_1(f)>\max_{i=2}^d \{\la_{i}(f)\}$.
Then $(X,f)$ satisfies the SZDO property.
\end{thm}

\subsection{Ergodic theory}\label{subsecergo}
Let $X$ be a variety over $\bk$. 
Denote by $|X|$ the underling set of $X$ with the constructible topology i.e. the topology on a  $X$ generated by the constructible subsets (see~\cite[Section~(1.9) and in particular (1.9.13)]{EGA-IV-I}).
In particular every constructible subset is open and closed.
This topology is finer than the Zariski topology on $X.$ Moreover $|X|$ is (Hausdorff) compact.

Denote by $\sM(|X|)$ the space of Radon measures on $X$ endowed with the weak-$\ast$ topology.
\begin{thm}\label{thmRadon}Every $\mu\in \sM(|X|)$ takes form 
$$\mu=\sum_{i\geq 0}a_i\delta_{x_i}$$
where $\delta_{x_i}$ is the Dirac measure at $x_i\in X$, $a_i\geq 0$.
\end{thm}
\begin{rem}Theorem \ref{thmRadon} is inspired by \cite[Theorem A]{Gignac2014}. In \cite[Theorem A]{Gignac2014}, Gignac worked on the Zariski topology, which is not Hausdorff. Here, we use the constructible topology systematically. We think that the constructible topology is the right topology for studying ergodic theory in algebraic dynamics.  For example, using constructible topology, we may avoid the conception of finite signed Borel measure used in  \cite[Theorem A]{Gignac2014}. Instead of it, we use the more standard notion of Radon measure.
\end{rem}

A sequence $x_n\in X, n\geq 0$ is said to be \emph{generic}, if every subsequence $x_{n_i}, i\geq 0$ is Zariski dense in $X.$
\begin{cor}\label{corgenericseqence}A sequence $x_n\in X, n\geq 0$ is generic if and only if 
$$\lim_{n\to \infty}\delta_{x_n}=\delta_{\eta},$$ where $\eta$ is the generic point of $X.$
\end{cor}

\medskip

Let $f: X\dashrightarrow X$ be a dominant  rational self-map. Set $|X|_f:=|X|\setminus (\cup_{i\geq 1}I(f^i)).$
Because every Zariski closed subset of $X$ is open and closed in the constructible topology,  $|X|_f$ is a closed subset of $|X|.$
The restriction of $f$ to $|X|_f$ is continuous. We still denote by $f$ this restriction.

\medskip
\subsubsection{DML problems}
Applying Corolary \ref{corgenericseqence}, the dynamical Moredell-Lang conjecture can be interpreted as the following equidistribution statement: 
\begin{dmlcon}[DML in form of equidistribution]
For $x\in X_f(\bk)$, if $O_f(x)$ is Zariski dense in $X$, then $$\lim_{n\to \infty}\delta_{f^n(x)}=\delta_{\eta}.$$
\end{dmlcon}
\begin{rem}
Here the assumption that $O_f(x)$ is Zariski dense in $X$ does not cause any problem. Because after replacing $x$ by some $f^m(x)$ and $f$ by a suitable iterate, we may assume that $\overline{O_f(x)}$ is irreducible. Then after replacing $X$ by $\overline{O_f(x)}$, we may assume that $O_f(x)$ is Zariski dense in $X$.
\end{rem}

\medskip

Using Theorem \ref{thmRadon}, we give a fast proof of the weak dynamical Mordell-Lang.
Same result was proved in \cite[Corollary 1.5]{Bell2015} (see also \cite[Theorem 2.5.8]{Favre2000a}, \cite[Theorem D, Theorem E]{Gignac2014}, \cite[Theorem 2]{Petsche2015}, \cite[Theorem 1.10]{Bell2020}). 
\begin{thm}[Weak DML]\label{thmdml}
Let  $x$ be a points $\in X_f(\bk)$ with $\overline{O_f(x)}=X.$ Let $V$ be a proper subvariety of $X$. Then $\{n\geq 0|\,\, f^n(x)\in V\}$ is of Banach density zero in $\Z_{\geq 0}$ i.e. for every sequence of intervals $I_n, n\geq 0$ in $\Z_{\geq 0}$ with $\lim_{n\to \infty}\# I_n=+\infty$, we have 
$$\lim_{n\to \infty}\frac{\#(\{n\geq 0|\,\, f^n(x)\in V\}\cap I_n)}{\#I_n}=0.$$
\end{thm}

We also prove the weak dynamical Mordell-Lang for coherent backward orbits.  A slightly weaker version was proved in \cite[Theorem F]{Gignac2014}.
This can be viewed as a weak version of \cite[Conjecture 1.5]{Xie2018}.
\begin{thm}[Weak DML for coherent backward orbits]\label{thmdmlback}
Let $x_n \in X_f(\bk), n\leq 0$ be a sequence of points such that $\overline{\{x_n, n\leq 0\}}=X$ and $f(x_n)=x_{n+1}$ for all $n\leq -1.$
Let $V$ be a proper subvariety of $X$. Then $\{n\leq 0|\,\, x_n\in V\}$ is of Banach density zero in $\Z_{\leq 0}$ 
\end{thm}

\subsubsection{Backward orbits}
Now assume that $f: X\to X$ is a flat and finite endomorphism. 
Let $d_f:=[\bk(X)/f^*\bk(X)]$ be topological degree of $f$. It is just the $(\dim X)$-th dynamical degree of $f$.

\medskip

Recall that for every $x\in X$, the multiplicity of $f$ at $x$ is $$m_f(x):=\dim_{\kappa(f(x))}(O_{X,x}/m_{f(x)}O_{X,x})\in \Z_{\geq 1}$$
where $O_{X,x}$ is viewed as an $O_{X,f(x)}$-module via $f$.
For every $x\in X$, we have $\sum_{y\in f^{-1}(x)}m_f(y)=d_f$ (see \cite[Theorem 2.4]{Gignac2014a}).

\medskip

In Section \ref{subsecfunct}, we define a natural pullback $f^*: \sM(X)\to \sM(X)$ which is continuous and for every $x\in X$,
$$f^*\delta_x=\sum_{y\in f^{-1}(x)}m_f(y)\delta_y.$$
We get the following equidistribution result.
\begin{thm}\label{thmequpullback}Let $f: X\to X$ be a flat and finite endomorphism. Let $x\in X(\bk)$ with $\overline{\cup_{i\geq 0}f^{-i}(x)}=X.$
Then for every sequence of intervals $I_n, n\geq 0$ in $\Z_{\geq 0}$ with $\lim_{n\to \infty}\# I_n=+\infty$, we have 
$$\lim_{n\to \infty}\frac{1}{\#I_n}(\sum_{i\in I_n}d_f^{-i}(f^i)^*\delta_{x})=\delta_{\eta}.$$
\end{thm}
\begin{rem}The assumption $\overline{\cup_{i\geq 0}f^{-i}(x)}=X$ is necessary. Otherwise, $$\frac{1}{\#I_n}(\sum_{i\in I_n}d_f^{-i}(f^i)^*\delta_{x}), n\geq 0$$ are supported on the proper closed subset $\overline{\cup_{i\geq 0}f^{-i}(x)}$ of $X.$
\end{rem}

\medskip



Applying Theorem \ref{thmequpullback}, we count the preimages of a point without multiplicities. 
\begin{thm}\label{thmcountprestr}Let $f: X\to X$ be a flat and finite endomorphism.
Assume that the field extension $\bk(X)/f^*\bk(X)$ is separable. Let $x\in X(\bk)$ be a point with $\overline{\cup_{i\geq 0}f^{-i}(x)}=X.$ 
For $c\in (0,1]$, $n\geq 0,$ define 
$$S^n_c:=\min\{\#S|\,\, S\subseteq f^{-n}(x),\,\, \sum_{y\in S}m_{f^n}(y)\geq cd_f^n\}.$$
Then for every $c\in (0,1]$, we have
$$\lim_{n\to \infty}(S^n_c)^{1/n}= d_f.$$
 \end{thm}

Taking $c=1$ in Theorem \ref{thmcountprestr}, we get the following corollary.
\begin{cor}\label{corcountpre}Let $f: X\to X$ be a flat and finite endomorphism.
If the field extension $\bk(X)/f^*\bk(X)$ is separable, then for every $x\in X(\bk)$ with $\overline{\cup_{i\geq 0}f^{-i}(x)}=X,$
 $$\lim_{n\to \infty}(\#f^{-n}(x))^{1/n}= d_f.$$
 \end{cor}

If the topological degree is large, we have the following stronger equidistribution result.
\begin{thm}\label{thmseqd}
Let $f: X\to X$ be a flat and finite endomorphism of a quasi-projective variety. Assume that 
\begin{equation}\label{equladdom}d_f:=\la_{\dim X}(f)>\max_{1\leq i\leq \dim X-1} \la_i.
\end{equation}
If the field extension $\bk(X)/f^*\bk(X)$ is separable, then for every $x\in X(\bk)$ with $\overline{\cup_{i\geq 0}f^{-i}(x)}=X,$
 $$\lim_{n\to \infty}d_f^{-n}(f^n)^*\delta_x=\delta_{\eta}.$$
 Moreover, for every irreducible subvariety $V$ of $X$ of dimension $d_V\leq \dim X-1$, $$\limsup_{n\to \infty}\#(f^{-n}(x)\cap V)^{1/n}\leq \la_{d_V}<d_f.$$
\end{thm}
Assumption \ref{equladdom} holds for polarized endomorphisms on projective varieties.
A similar statement for polarized endomorphisms can be fund in \cite[Theorem 5.1]{Gignac2014a}.
See \cite{Guedj2005,Dinh2015} for according result for complex topology.

\medskip

Theorem \ref{thmseqd} is not true without Assumption \ref{equladdom}. 
\begin{exe}
Under the notation of Example \ref{exenotdml}. Set $g:=f^{-1}.$ Then $\la_i(g)=1, i=0,1,2.$
Denote by $1_{V}$ the characteristic function of $V$. Since $V$ is open and closed in $|\A^2|$, 
$1_{V}$ is continuous. 
We have $$\lim_{n\to \infty}\int 1_V(g^{-p^n})^*\delta_{e}=\lim_{n\to \infty}1_V(f^{p^n}(e))=1\neq 0=\int 1_V\delta_{\eta}.$$
\end{exe}

\subsection{Relation to Berkovich spaces}
 We will see in Section \ref{subsecberko}, $|X|$ can be viewed as a closed subset of the Berkovich analytification $X^{\an}$ of $X$ w.r.t the trivial norm on $\bk$.  
 So the statements in ergodic theory on $|X|$ can be translated to statements on $X^{\an}.$
 See the translation of Corollary \ref{corgenericseqence} and Theorem \ref{thmseqd} in Section \ref{subsecberko}.
 
Using reduction map, we may also use ergodic theory w.r.t. the constructible topology to study endomorphisms on Berkovich spaces with good reduction.
In Section \ref{subsectionreduction},
we apply Theorem \ref{thmseqd} to get an equidistribution result for endomorphisms of large topological degree with good reduction.
\subsection{Notation and Terminology}
\begin{points}
\item[$\bullet$] For a set $S$, denote by $\# S$ the cardinality of $S.$
\item[$\bullet$] A \emph{variety} is an irreducible separated scheme of finite type over a field.
A \emph{subvariety} of a variety $X$ is a closed subset of $X.$
\item[$\bullet$] For a variety $X$ (resp. a rational self-map $f: X\dashrightarrow Y$) over a field $k$ and a subfield $K$ of $k$, we say that $X$ (resp. $f$) is \emph{defined over $K$} if there is a variety $X_K$ (resp. a rational map $f_K$) over $K$ such that $X$ (resp. $f$) is the base change by $k$ of $X$ (resp. $f$).
\item[$\bullet$]For a rational map $f: X\dashrightarrow Y$ between varieties. Denote by $I(f)$ the indeterminacy locus of $f$. 
\item[$\bullet$] For a dominant rational self-map $f: X\dashrightarrow X$ between varieties, a subvariety $V$ of $X$ is said to be \emph{$f$-invariant} if 
$I(f)$ does not contain any irreducible component of $V$ and 
$f(V)\subseteq V.$ 
\item[$\bullet$] For a projective variety $X$, $N^i(X)$ is the the group of numerical $i$-cycles of $X$ and $N^i(X)_{\R}:=N^i(X)\otimes \R.$ 
\item[$\bullet$] For two Cartier $\R$-divisors $D_1,D_2$, write $D_1\equiv D_2$ if $D_1,D_2$ are numerically equivalent.
\item[$\bullet$] For a field extension $k/K$, $\trd_Kk$ is the transcendence degree of $k/K.$
\end{points}

\subsection*{Acknowledgement}
I would like to thank Xinyi Yuan.  Section \ref{sectionergodictheory} of this paper  is motivated by some interesting discussion with him.

\section{Dynamical degree and arithmetic degree}

\subsection{The dynamical degrees}\label{subsectiondydeg}
In this section we recall the definition and some basic facts on the dynamical degree.

Let $X$ be a variety over $\bk$ and $f: X\dashrightarrow X$ a dominant rational self-map.
Let $X'$ be a normal projective variety which is birational to $X$.
Let $L$ be an ample (or just nef and big) divisor on $X'$.
Denote by $f'$ the rational self-map of $X'$ induced by $f$.

For $i=0,1,\dots,\dim X$, and $n\geq 0$,  $(f'^n)^*(L^i)$ is the $(\dim X-i)$-cycle on $X'$ as follows: let $\Gamma$ be a normal projective variety with a birational morphism $\pi_1\colon\Gamma\to X'$ and a morphism $\pi_2\colon\Gamma\to X'$ such that $f'^n=\pi_2\circ\pi_1^{-1}$.
Then $(f'^n)^*(L^i):= (\pi_1)_*\pi_2^*(L^i)$.
The definition of $(f'^n)^*(L^i)$ does not depend on the choice of $\Gamma$, $\pi_1$ and $\pi_2$.
The $i$-th \textit{dynamical degree} of $f$ is
$$
	\la_i(f):=\lim_{n\to\infty}((f'^n)^*(L^i)\cdot L^{\dim X-i})^{1/n}.
$$
The limit converges and does not depend on the choice of $X'$ and $L$
\cite{Russakovskii1997, Dinh2005, Truong2020,Dang2020}.
Moreover, if $\pi: X\dashrightarrow Y$ is a generically finite and dominant rational map between varieties and $g\colon Y\dashrightarrow Y$ is a rational self-map such that $g\circ\pi=\pi\circ f$, then $\la_i(f)=\la_i(g)$ for all $i$;
for details, we refer to \cite[Theorem 1]{Dang2020} (and the projection formula), or Theorem 4 in its arXiv version \cite{Dang}.

The following result is easy  when $\bk$ is of characteristic 0 and $Z\not\subseteq \Sing X$.

\begin{pro}\cite[Proposition 3.2]{Jia2021}\label{p:dyn_sub_var}
Let $X$ be a variety over $\bk$ and $f\colon X\dashrightarrow X$ a dominant rational self-map.
Let $Z$ be an irreducible subvariety in $X$ which is not contained in $I(f)$ such that $f|_Z$ induces a dominant rational self-map of $Z$.
Then $\la_i(f|_Z)\leq \la_i(f)$ for $i=0,1,\dots,\dim Z$.
\end{pro}

\subsection{Arithmetic degree}
The arithmetic degree was defined in \cite{Kawaguchi2016} over a number field or a function field of characteristic zero.
In this section we extend this definition to the case over function field of positive characteristic and we prove some basic fact of it.

Let $\bk=\overline{K(B)}$, where $K$ is an algebraically closed field and $B$ is a smooth projective curve. 

\subsubsection{Weil height}
Let $X$ be a normal and projective variety over $\bk.$ 
For every $L\in \Pic(X)$, we denote by $h_L: X(\bk)\to \R$ a Weil height associated to $L$ and the function field $K(B)$. It is unique up to adding a bounded function. 

\begin{exe}\label{exekbheight}
Assume that $X$ is defined over $K(B)$ i.e. there is a projective morphism $\pi: X_B\to B$ where $X_B$ is normal, projective and geometric generic fiber of $\pi$ is $X$.
Assume that there is a line bundle $L_B$ on $X_B$ whose restriction on $X$ is $L$. In this case, for every $x\in X(\bk)$, we may take $h_L$ to be 
$$h_{(X_B, L_B)}(x)=[K(B)(x):K(B)]^{-1}(\overline{x}\cdot L),$$
where $\overline{x}$ is the Zariski closure of $x$ in $X_B.$
\end{exe}

Keep the notations in Example \ref{exekbheight}.
Let $b$ be a point in $B(K).$ It induces a norm $|\cdot|_b$ on $K(B)$. Denote by $K(B)_b$ the completion of $K(B)$ w.r.t. $|\cdot|_b$.
Denote by $\C_b$ the completion of $\overline{K(B)_b}.$ Every field embedding $\tau: \bk=\overline{K(B)}\hookrightarrow \C_b$ induces an embedding 
$\phi_{\tau}: X(\bk)\hookrightarrow X(\C_b).$ On $X(\C_b)$, we have a natural $b$-adic topology induced by  $|\cdot|_b$. 

\begin{rem}\label{remopenred}Let $x_b$ be a point in $X_b$. Then $x_b$ defines a nonempty open subset $U_{x_b}$ consisting of all points in $X(\C_b)$ whose reduction is $x_b\in X_b(K).$
Then for every $x\in \phi_{\tau}^{-1}(U_{x_b})$,  $x_0$ is contained in the Zariski closure of $x$ in $X_B.$
\end{rem}
\begin{lem}\label{lemlocheig}
There is $d\geq 1$ such that for every $b\in B(K),$ every nonempty $b$-adic open subset of $U\subseteq X(\C_b),$
and  every $l\geq 1$, there is $x\in X(\bk)$ such that $\deg(x)\leq d$ and $h_L(x)\geq l$.
\end{lem}

\proof By Noether normalization lemma, we only need to prove the lemma when $X=\P^N$ and $L=O(1).$
After replace $K(B)$ by a finite extension,  a changing of coordinates, we may assume that $0\in U.$
We may assume that $h_L$ is the naive height on $\P^N$ i.e. the height defined by the model $(\P^N_B, O_{\P^N(B)}(1)).$
Pick any rational function $g\in K(B)\setminus \{0\}$ with $g(b)=0.$ Then for $n\geq 1$, $x_n:=(g^n,\dots, g^n)\in \A^N(K(B)).$
We have $h_L(x_n)\to \infty$ as $n\to \infty$ and $\phi_{\tau}(x_n)\to 0$ in the $b$-adic topology. This concludes the proof.
\endproof

\subsubsection{Admissible triples.}
As in \cite{Jia2021}, we define an \textit{admissible triple} to be $(X,f,x)$ where $X$ is a quasi-projective variety over $\bk$, $f\colon X\dashrightarrow X$ is a dominant rational self-map and $x\in X_f(\bk)$.

We say that $(X,f,x)$ \textit{dominates} (resp.~\textit{generically finitely dominates}) $(Y,g,y)$ if there is a dominant rational map (resp.~generically finite and dominant rational map) $\pi\colon X\dashrightarrow Y$ such $\pi\circ f=g\circ\pi$, $\pi$ is well defined along $O_f(x)$ and $\pi(x)=y$.

We say that $(X,f,x)$ is \textit{birational} to $(Y,g,y)$ if there is a birational map $\pi\colon X\dashrightarrow Y$ such $\pi\circ f=g\circ\pi$ and if there is a Zariski dense open subset $V$ of $Y$ containing $O_g(y)$ such that $\pi|_U: U:=\pi^{-1}(V)\to V$ is a well-defined isomorphism and $\pi(x)=y$.
In particular, if $(X,f,x)$ is birational to $(Y,g,y)$, then $(X,f,x)$ generically finitely dominates $(Y,g,y)$.

\begin{rem}
\leavevmode
\begin{enumerate}
	\item If $(X,f,x)$ dominates $(Y,g,y)$ and if $O_f(x)$ is Zariski dense in $X$, then $O_g(y)$ is Zariski dense in $Y$.
	Moreover, if $(X,f,x)$ generically finitely dominates $(Y,g,y)$, then $O_f(x)$ is Zariski dense in $X$ if and only if $O_g(y)$ is Zariski dense in $Y$.
	\item Every admissible triple $(X,f,x)$ is birational to an admissible triple $(X',f',x')$ where $X'$ is projective.
	Indeed, we may pick $X'$ to be any projective compactification of $X$, $f'$ the self-map of $X'$ induced from $f$, and $x'=x$.
	\end{enumerate}
\end{rem}


\subsubsection{The set $A_f(x)$.}
As in \cite{Jia2021}, we will associate to an admissible triple $(X,f,x)$ a subset $$A_f(x)\subseteq [1,\infty].$$
\begin{rem}
We will show in Proposition \ref{proupboundarth} that $A_f(x)\subseteq [1,\la_1(f)].$
\end{rem}

We first define it when $X$ is projective. Let $L$ be an ample divisor on $X$, we define
$$A_f(x)\subseteq [1,\infty]$$
to be the limit set of the sequence $(h_L^+(f^n(x)))^{1/n}$, $n\geq 0$, where $h_L^+(\cdot):=\max\{h_L(\cdot),1\}$.

The following lemma was proved in \cite[Lemma 3.8]{Jia2021} when $\bk=\overline{\Q}$, but its proof still works our case.
It shows that the set $A_f(x)$ does not depend on the choice of $L$ and is invariant in the birational equivalence class of $(X,f,x)$.

\begin{lemma}\cite[Lemma 3.8]{Jia2021}\label{lemsingwilldef}
Let $\pi\colon X\dashrightarrow Y$ be a dominant rational map between projective varieties.
Let $U$ be a Zariski dense open subset of $X$ such that $\pi|_U\colon U\to Y$ is well-defined.
Let $L$ be an ample divisor on $X$ and $M$ an ample divisor on $Y$.
Then there are constants $C\geq 1$ and $D>0$ such that for every $x\in U$, we have
\begin{equation}\label{equationdomineq1}
	h_M(\pi(x))\leq Ch_L(x)+D.
\end{equation}

Moreover if $V:=\pi(U)$ is open in $Y$ and $\pi|_U\colon U\to V$ is an isomorphism, then 
 there are constants $C\geq 1$ and $D>0$ such that for every $x\in U$, we have
\begin{equation}\label{equationbirdomineq}
	C^{-1}h_L(x)-D\leq h_M(\pi(x))\leq Ch_L(x)+D.
\end{equation}
\end{lemma}

Now for every admissible triple $(X,f,x)$, we define $A_f(x)$ to be $A_{f'}(x')$ where $(X',f',x')$ is an admissible triple which is birational to $(X,f,x)$ such that $X'$ is projective.
By Lemma~\ref{lemsingwilldef}, this definition does not depend on the choice of $(X',f',x')$.

\subsubsection{The arithmetic degree.}\label{subsec_arithdeg}
We define (see also \cite{Kawaguchi2016}):
\[
	\overline{\alpha}_f(x):=\sup A_f(x),\qquad\underline{\alpha}_f(x):=\inf A_f(x).
\]
We say that $\alpha_f(x)$ is well-defined and call it the \textit{arithmetic degree} of $f$ at $x$, if $\overline{\alpha}_f(x)=\underline{\alpha}_f(x)$;
and, in this case, we set
\[
	\alpha_f(x):=\overline{\alpha}_f(x)=\underline{\alpha}_f(x).
\]
By Lemma~\ref{lemsingwilldef}, if $(X,f,x)$ dominates $(Y,g,y)$, then $\overline{\alpha}_f(x)\geq \overline{\alpha}_g(y)$ and $\underline{\alpha}_f(x)\geq\underline{\alpha}_g(y)$.

Applying Inequality~\eqref{equationdomineq1} of Lemma~\ref{lemsingwilldef} to the case where $Y=X$ and $M=L$, we get the following trivial upper bound:
let $f\colon X\dashrightarrow X$ be a dominant rational self-map, $L$ any ample line bundle on $X$ and $h_L$ a Weil height function associated to $L$;
then there is a constant $C\geq 1$ such that for every $x\in X\setminus I(f)$, we have
\begin{equation}\label{equationtrivialupper}
	h_L^+(f(x))\leq Ch_L^+(x).
\end{equation}
For a subset $A\subseteq [1,\infty)$, define $A^{1/\ell}:= \{a^{1/\ell}\mid a\in A\}$.

We have the following simple properties, where the second half of \ref{eq:alpha_pow} used Inequality~\eqref{equationtrivialupper}.
\begin{pro}\label{probasicaf}We have:
\begin{enumerate}
	\item $A_f(x)\subseteq [1,\infty)$.
	\item $A_f(x)=A_f(f^{\ell}(x))$, for any $\ell\geq 0$.
	\item \label{eq:alpha_pow}
	$A_{f}(x)=\bigcup_{i=0}^{\ell-1}(A_{f^{\ell}}(f^i(x)))^{1/\ell}$.
	In particular, $\overline{\alpha}_{f^{\ell}}(x)=\overline{\alpha}_{f}(x)^{\ell}$, $\underline{\alpha}_{f^{\ell}}(x)=\underline{\alpha}_{f}(x)^{\ell}$.
\end{enumerate}
\end{pro}

The following lemma is easy.
\begin{lemma}
\label{lem_subvar}
Let $f\colon X\dashrightarrow X$ be a dominant rational self-map of a projective variety $X$ and $W\subseteq X$ an $f$-invariant subvariety.
Then $X_f(\bk)\cap W(\bk)\subseteq W_{f|_W}(\bk)$ and for every $x\in X_f(\bk)\cap W(\bk)$,
$\alpha_{f|_W}(x)=\alpha_f(x).$
\end{lemma}

When $\bk=\overline{\Q}$, the next result was proved in \cite[Theorem 1.4]{Matsuzawa2020a} in the smooth case and in \cite[Proposition 3.11]{Jia2021} in the singular case.
The proof here in the function field case is much easier.

\begin{pro}[Kawaguchi-Silverman-Matsuzawa's upper bound]\label{proupboundarth}
For every admissible triple $(X,f,x_0)$, we have $\overline{\alpha}_f(x_0)\leq \la_1(f)$.
\end{pro}

\begin{proof}
We may assume that $X$ is projective. Set $d:=\dim X.$
After replacing $f$ by a suitable iteration and $x_0$ by $f^n(x_0)$ for some $n\geq 0$ and noting that $\la_1(f^n)=\la_1(f)^n$ and by Proposition \ref{probasicaf},
 we may assume that the Zariski closure $Z_f(x_0)$ of $O_f(x_0)$ is irreducible.
 By Proposition~\ref{p:dyn_sub_var} and Lemma~\ref{lem_subvar}, we may replace 
 $X$ by $Z_f(x_0)$ and  assume that $O_f(x_0)$ is Zariski dense in $X$.

Assume that $X$ is defined over $K(B)$ i.e. there is a projective morphism $\pi: \sX\to B$ where $\sX$ is projective, normal and geometric generic fiber of $\pi$ is $X$.
Pick an ample line bundle $L_B$ on $\sX$ and let  $L$ be its restriction to $X$. 
We take the Weil height $h_L: X(\bk)\to \R$ as follows: for every $x\in X(\bk)$,  
$$h_L(x):=h_{(\sX, L_B)}(x)=[K(B)(x):K(B)]^{-1}(\overline{x}\cdot \sL).$$
We may assume that $x_0$ is defined over $K(B)$. 

Let $F: \sX\dashrightarrow \sX$ be the rational self-map over $B$ induced by $f.$
The relative dynamical degree formula \cite[Theorem 4]{Dang}, shows that $$\la_1(F)=\max\{1,\la_1(f)\}=\la_1(f).$$
So for every $r>0$, there is $C_r>0$ such that for every $n\geq 0$, 
\begin{equation}\label{equcnln}((F^n)^*L_B\cdot L_B^d)\leq C_r (\la_1(f)+r)^n.
\end{equation}

Let $\sI$ be the ideal sheaf of $\overline{x_0}$ on $\sX.$
After replacing $L_B$ be a suitable multiple, we may assume that  
$\sL\otimes \sI$ is globally generated.
For every $n\geq 0$, there are divisors $H_i, i=0,\dots,d$ in $|L_B|$ such that 
$\dim H_1\cap \dots \cap H_d=1$ and containing $\overline{x_0}$ as an irreducible component.

Set $V_n:=H_1\cdot \dots \cdot H_d$. 
Let $\Gamma$ be a normal projective variety with a birational morphism $\pi_1\colon\Gamma\to \sX$ and a morphism $\pi_2:\Gamma\to \sX$ such that $F^n=\pi_2\circ\pi_1^{-1}$.
Write $(\pi_1)^{\#}\overline{x_0}$ the strict transform of $V^n$  $\overline{x_0}$ by $\pi_1^N.$
Then $(\pi_1)^{\#}\overline{x_0}$ is an irreducible component of $\cap_{i=1}^d(\pi_1^*H_i).$
In $N^1(\Gamma)$, we have $\pi_1^*V_n=\pi_1^*H_1\cdot\dots \cdot \pi^*H_d.$
By \cite[Lemma 3.3]{Jia2021}, $\pi_1^*V_n-(\pi_1)^{\#}\overline{x_0}$ is pseudo-effective.
Then we have
$$h_L(f^n(x_0))=(\overline{f^n(x_0)}\cdot L_B)=((\pi_1)^{\#}\overline{x_0}\cdot \pi_2^*L_B)$$
$$\leq (\pi_1^*H_1\cdot \dots \cdot \pi^*_1H_d\cdot \pi_2^*L_B)=((F^n)^*L_B\cdot L_B^d).$$
 $$\leq C_r (\la_1(f)+r)^n.$$
 It follows that 
 $$\overline{\alpha}_f(x_0)=\limsup_{n\to \infty}h_L(f^n(x_0))^{1/n}\leq \lim_{n\to \infty}(C_r (\la_1(f)+r)^n)^{1/n}=\la_1(f)+r.$$
 Letting $r\to \infty$, we conclude the proof.
\end{proof}
%

\subsection{Canonical height}
Let $X$ be a normal projective variety and $f: X\to X$ a surjective endomorphism.

Let $A$ be an ample divisor of $X$, denote by $h_A$ a Weil height on $X(\bk)$ associated to $A$ with $h_A\geq 1.$ 
\begin{pro}\label{procanheight}
Let $D$ be a nonzero Cartier $\R$-divisor such that $f^*D\equiv\beta D$ where $\beta> \la_1(f)^{1/2}.$
Let $[D]\in N^1(X)_{\R}$ be the numerical class of $D.$
Then 
for every $x\in X(\bk)$, the limit
$h_{[D]}^+(x):=\lim_{n\to \infty}h_{D}(f^n(x))/\beta^n$
exist, only depend on the numerical class $[D]$ and
satisfies the following properties:
\begin{points}
\item $h_{[D]}^+=h_{D}+O(h_A^{1/2})$;
\item $h_{[D]}^+\circ f=\beta h^+$.
\end{points}
\end{pro}
\proof
This result was proved in \cite[Theorem 5]{Kawaguchi2016}  in characteristic zero. 
The proof presented here is the same as \cite[Theorem 5]{Kawaguchi2016}, but slightly shorter.  

By \cite[Proposition B.3]{Matsuzawa2020a}, there is $C>0$ such that for every $x\in X(\bk)$,
$$|h_D(f(x))-\beta h_D(x)|\leq Ch_A(x)^{1/2}.$$
Pick $\mu\in (\la_1(f)^{1/2}, \beta),$ by Proposition \ref{proupboundarth}, for every $x\in X(\bk)$, there is $C_x>0$ such that ,
$$h_A(f^n(x))\leq C_x\mu^{2n}h_A(x).$$
Then we have $$|h_{D}(f^n(x))/\beta^n-h_{D}(f^{n-1}(x))/\beta^{n-1}|=\beta^{-n}|h_{D}(f^n(x))-\beta h_{D}(f^{n-1}(x))|$$
$$\leq \beta^{-n}Ch_A(f^{n-1}(x))^{1/2}\leq \beta^{-n}CC_x^{1/2}\mu^nh_A(x)^{1/2}=CC_x^{1/2}(\mu/\beta)^nh_A(x)^{1/2}.$$
Since $0<\mu/\beta<1$,  
$$h_{[D]}^+(x)=h_D(x)+\sum_{n\geq 1}(h_{D}(f^n(x))/\beta^n-h_{D}(f^{n-1}(x))/\beta^{n-1})$$
converges and 
$$|h_{[D]}^+(x)-h_D(x)|\leq \sum_{n\geq 1}|h_{D}(f^n(x))/\beta^n-h_{D}(f^{n-1}(x))/\beta^{n-1}|$$
$$\leq (\sum_{n\geq 1}CC_x^{1/2}(\mu/\beta)^n)h_A(x)^{1/2}=O(h_A(x)^{1/2}).$$
Then we get (i).
The statement (ii) follows from the definition.

For $D'\equiv D$, by \cite[Proposition B.3]{Matsuzawa2020a}, there is $B>0$ such that for every $x\in X(\bk)$,
$$|h_{D'}(x)-h_D(x)|\leq Bh_A(x)^{1/2}.$$
Then
$$|h_{[D']}^+(x)-h_{[D]}^+(x)|:=\lim_{n\to \infty}|h_{D'}(f^n(x))-h_{D}(f^n(x))|/\beta^n$$
$$\leq \limsup_{n\to \infty}Bh_A(f^n(x))^{1/2}/\beta^n\leq \limsup_{n\to \infty}BC_xh_A(x)^{1/2}(\mu/\beta)^n=0,$$
which concludes the proof.\endproof

The following was proved in \cite[Lemma 9.1]{Matsuzawa2018} when $\bk=\overline{\Q}$ and $X$ is smooth.
After replacing \cite[Theorem 5]{Kawaguchi2016} by Proposition \ref{procanheight},
\cite[Lemma 9.1]{Matsuzawa2018} is still valid when $\bk=K(B)$ and $X$ is singular.
\begin{pro}\label{proextarhitd}
Assume that $\la_1(f)>1$.
Let $D\not\equiv 0$ be a nef $\R$-Cartier divisor on $X$ such that $f^*D\equiv \la_1(f)D$.
Let $V\subseteq X$ be a subvariety of positive dimension such that $(D^{\dim V}\cdot V)>0$.
Then there exists a nonempty open subset $U\subseteq V$ and a set $S\subseteq U(\bk)$ of bounded height such that for every $x\in U(\bk)\setminus S$ we have $\alpha_f(x)=\la_1(f)$.
\end{pro}

\begin{cor}\label{cordenseal}Keep the notation in Proposition \ref{proextarhitd}.
For every Zariski dense open subset $U$ of $X$, there is $x\in U(\bk)$ such that $\alpha_f(x)=\la_1(f)$ and $O_f(x)\subseteq U$.
\end{cor}
\proof[Proof of Corollary \ref{cordenseal}]
We may assume that $X$ is normal and $X,f,D$ and $U$ are defined over $K(B).$
There is a normal and projective $B$-scheme $\pi: X_B\to B$ and a rational self-map $f_B: X_B\dashrightarrow X_B$ over $B$ such that the geometric generic fiber of $(X_B,f_B)$ is $(X,f).$
Let $b$ be a general point of $B(K)$ and denote by $(X_b,f_b)$ the fiber of $(X_B,f_B)$ above $b$. 
Then $f_b$ is an endomorphism of $X_b.$ Set $Z:=X\setminus U$. Let $Z_B$ be the Zariski closure of $Z$ in $X_B.$
Then $U_b:=X_b\setminus Z_B$. 
By Proposition \ref{proxfnonempty} (see Section \ref{subsectionexweldef} for its proof), there is $x_b\in (U_b)_{f_b|_{U_b}}(K).$
Let $M$ be a very ample line bundle on $X_B.$
Taking $W_B$ to be the intersection of $\dim X-1$ general elements of $|10M|$ of $X_B$ passing through $x_b.$
By \cite[Theorem 0.4]{Benoist2011},  $W_B$ is irreducible.
Let $W\subseteq X$ be the generic fiber of $W_B$. It is of pure dimension $1$.
Then $(W\cap D)>0.$ 
Because $W_B$ is irreducible, for every irreducible component $W'$ of $W$,
$(W'\cdot D)>0.$ By Lemma \ref{lemlocheig} and Remark \ref{remopenred}, there are $x_n\in W'(\bk), n\geq 0$ such that $x_b\in \overline{\{x_n\}}$ and the height of $x_n$ tends to $+\infty$.
Because $O_{f_b}(x_b)\subseteq U$, $O_f(x_n)\subseteq U$ for all $n\geq 0.$
By Proposition \ref{proextarhitd}, for $n>>0$,  we have $x_n\in V(\bk)\cap U$ and $\alpha_f(x_n)=\la_1(f)$.
\endproof

\section{Proof of Theorem \ref{dmlsurface}}
This proof mixes the ideas from \cite{Xie2014} and \cite{Lesieutre2021}.
\subsection{Reduce to the smooth case}
By \cite{Lipman1978}, there is a minimal desingularization $\pi:X'\to X$. Then one may lift $f$ to an automorphism $f'$ of $X'.$
The following lemma allows us to replace $(X,f)$ by $(X',f')$ and assume that $X$ is smooth.

\begin{lem}\label{lemredudesing}If $(X',f')$ satisfies the DML property, then $(X,f)$ satisfies the DML property.
\end{lem}

\proof
Assume that $(X',f')$ satisfies the DML property. 
We only need to prove the following statement:
for every $x\in X(\bk)$ and an irreducible curve $C\subseteq X(\bk)$, if $O_f(x)\cap C$ is infinite, then $C$ is $f$-periodic.

Pick $x'\in \pi^{-1}(x)(\bk)$. There is an irreducible component $C'$ of $\pi^{-1}(C)$ such that $O_{f'}(x')\cap C'$ is infinite.
We have $\dim C'\leq 1.$
If $\pi(C')\neq C$, then $\pi(C')$ is a point. Then $x=\pi(x')$ is periodic. So $\pi(C')= C$ and $\dim C'=1.$
Since $(X',f')$ satisfies the DML property, $C'$ is $f'$-periodic. So $C=\pi(C')$ is $f'$-periodic.
\endproof

\subsection{Numerical geometry}\label{subsectionnumgeo}
Set $\la:=\la_1(f)>1$. There is a nef class $\theta^*\in N^1(X)_{\R}\setminus \{0\}$ such that $f^*\theta^*=\la\theta^*.$
By projection formula $\la_1(f^{-1})=\la.$ So there is a nef class $\theta^*\in N^1(X)_{\R}\setminus \{0\}$ such that $(f^{-1})^*\theta_*=\la\theta_*.$
Then $f^*\theta_*=\la^{-1}\theta_*.$ 
Since $\la^2({\theta^*}^2)=({f^*\theta^*}^2)=({\theta^*}^2),$  we get $({\theta^*}^2)=0.$ Similarly, $({\theta_*}^2)=0.$
By Hodge index theorem, $(\theta^*\cdot \theta_*)>0.$ It follows that 
$(\theta^*+\theta_*)^2>0.$ So $\theta^*+\theta_*$ is big and nef.

Set $H:=\{\alpha\in \N^1(X)_{\R}|\,\, (\theta^*\cdot \alpha)=(\theta_*\cdot \alpha)=0\}.$  It is clear that $\N^1(X)_{\R}=\R\theta^*\oplus \R\theta_*\oplus H$ and $f^*H=H.$
By Hodge index theorem, the intersection form on $H$ is negative define. Since $f^*$ preserves the intersection form, all eigenvalues of $f^*|_H$ are of norm $1$.

Since $f^*$ is an automorphism of the lattes $N^1(X)\subseteq N^1(X)_{\R}$, all eigenvalues of $f^*: N^1(X)_{\R}\to N^1(X)_{\R}$ are algebraic integers.
In particular both $\la$ and $\la^{-1}$ are algebraic integers.

\begin{lem}\label{lemlaoneconj}There is $\sigma\in \Gal(\overline{\Q}/\Q)$ such that $\sigma(\la)=\la^{-1}.$
\end{lem}
\proof[Proof of Lemma \ref{lemlaoneconj}]
Since $\la_1$ is an algebraic integer with $|\la|>1$, by product formula, there is $\sigma\in \Gal(\overline{\Q}/\Q)$ such that 
$|\sigma(\la_1)|<1.$ Because $\sigma(\la_1)$ is an eigenvalue of $f^*$ and $\la_1^{-1}$ is the unique eigenvalue of $f^*$ with norm $<1,$
we have $\sigma(\la_1)=\la_1^{-1}$.
\endproof 

Then $f^*\sigma(\theta^*)=\sigma(f^*\theta^*)=\sigma(\la)\sigma(\theta^*)=\la^{-1}\sigma(\theta^*).$
So there is $c>0$ such that $\theta_*=c\sigma(\theta^*).$ After replacing $\theta_*$ by $c^{-1}\theta_*$, we may assume that 
$\sigma(\theta^*)=\theta_*.$

\begin{cor}\label{corcurudst}For every curve $C$ of $X$, $(\theta^*\cdot C)=0$ if and only if $(\theta_*\cdot C)=0.$
\end{cor}
\proof[Proof of Corollary \ref{corcurudst}]
The subspace $P:=\{\alpha\in N^1(X)_{\C}|\,\, (\alpha\cdot C)=0\}$ is a hyperplane of $N^1(X)_{\C}$ defined over $\Q.$
We have $\sigma(P)=P.$
Embed $N^1(X)_{\R}$ in $N^1(X)_{\C}$. Then $\theta^*\in P$ if and only if $\theta_*=\sigma(\theta^*)\in \sigma(P)=P.$
\endproof

\subsection{Canonical height}
In this section, we assume  
\begin{points}
\item either $\bk=\overline{\Q}$;
\item or there is an algebraically closed subfield $K\subseteq \bk$, a curve $B$ over $K$, such that $X$ and $f$ are defined over $K(B)$ and $\bk=\overline{K(B)}.$
\end{points}

Let $A$ be an ample divisor of $X$, denote by $h_A$ a Weil height on $X(\bk)$ associated to $A$ with $h_A\geq 1.$
Pick $\R$-divisors $D^{*}$ and $D_*$ with numerical classes $\theta^*, \theta_*$.
By \cite[Theorem 5]{Kawaguchi2016} and \cite{Kawaguchi2020} in characteristic zero and Proposition \ref{procanheight} in positive characteristic, for every $y\in X(\bk)$, the limits 
$$h^+(y):=\lim_{n\to \infty}h_{D^*}(f^n(y))/\la^n$$
and 
$$h^-(y):=\lim_{n\to \infty}h_{D_*}(f^{-n}(y))/\la^n$$
exist, do not depend on the choice of $D^{*}$, $D_*$, $h_{D^*}$ and $h_{D_*},$ and satisfies the following properties:
\begin{points}
\item $h^+=h_{D^*}+O(h_A^{1/2})$, $h^-=h_{D_*}+O(h_A^{1/2})$;
\item $h^+\circ f=\la h^+$ and $h^-\circ f=\la^{-1}h^-.$
\end{points}

\begin{lem}\label{lemcomheigonc}Let $C$ be an irreducible curve of $X$ such that $(C\cdot \theta_*)>0.$ Then for every $M\geq 0$, there is $M'\geq 0$, such that 
$$\{y\in C(\bk)|\,\, h^-(y)\leq M\}\subseteq \{y\in C(\bk)|\,\, h_A(y)\leq M'\}.$$
\end{lem}

\proof[Proof of Lemma \ref{lemcomheigonc}]
There is $d>0$, such that $$h^{-}\geq h_{D_*}-dh_A^{1/2}.$$
Pick $a>0$ such that $a(D_*\cdot C)>(A\cdot C).$
Then there is $b>0$ such that for every $y\in C$, 
$$ah_{D^*}(y)+b\geq h_{A}(y).$$
So for every $y\in C,$
$$h^{-}(y)\geq a^{-1}(h_{A}(y)-b)-dh_A^{1/2}(y).$$
If $h^-(y)\leq M$, we get 
$$M\geq a^{-1}(h_{A}(y)-b)-dh_A^{1/2}(y)=(a^{-1}h_A^{1/2}(y)-d)h_A^{1/2}(y)-a^{-1}b.$$
This implies that 
$$h_A^{1/2}(y)\leq \max\{ad, aM+b+ad\}=aM+b+ad.$$
Then we get $h_A(y)\leq (aM+b+ad)^2.$
\endproof

\subsection{The case $(C\cdot \theta_*)>0$}
\begin{lem}\label{lemcthqezerfinite}Let $C$ be an irreducible curve of $X$ such that $(C\cdot \theta_*)>0.$
For every $x\in X(\bk)$, $O_f(x)\cap C$ is finite.
\end{lem}

\proof[Proof of Lemma \ref{lemcthqezerfinite}]
Let $\F$ be the minimal algebraically closed subfield of $\bk$. So $\F=\overline{\Q}$ if $\Char\, \bk=0$ and $\F=\overline{\F_p}$ when $\Char\, \bk=p>0.$
There is an algebraically closed subfield $\bk'$ of $\bk$ with $\trd_{\F}\bk'<\infty$ such that $X,f,C$ and $x$ are defined over $\bk'$.  After replacing $\bk$ by $\bk'$, we may assume $\trd_{\F}\bk<\infty$. 
Now we prove Lemma \ref{lemcthqezerfinite} by induction on $\trd_{\F}\bk.$

\medskip

When $\bk=\overline{\F_p}$ for some prime $p>0,$  $O_f(x)$ is finite.
Then Lemma \ref{lemcthqezerfinite} holds.

\medskip

Assume $\bk=\overline{\Q}.$ Set $I:=\{i\geq 0|\,\, f^i(x)\in C\}.$
For every $i\geq I$, $h^-(f^i(x))=\la^{-i}h^-(x)\leq h^-(x).$ By Lemma \ref{lemcomheigonc}, there is $M>0$ such that 
$h_A(f^i(x))<M$ for every $i\in I.$ We conclude the proof by the Northcott property.

\medskip

Now we may assume that $\trd_{\F}\bk\geq 1.$ There is an algebraically closed subfield $K\subseteq \bk$, a smooth irreducible projective curve $B$ over $K$, such that $X$, $f$, $C$ and $x$ are defined over $K(B)$ and $\bk=\overline{K(B)}.$

There is a projective morphism $\pi: \sX\to B$ whose geometric generic fiber is $X$. The automorphism $f$ extends to a birational self-map $f_B: \sX\dashrightarrow \sX$ over $B$.
Let $A_B$ be an ample divisor on $\sB$.  Let $C_B$ be the Zariski closure of $C$ in $\sX$. 
Let $A$ be the restriction of $A_B$ on the generic fiber $X.$ There is a nonempty open subset $U$ of $B$, such that $\pi$ is smooth above $U$ and $f_B|_{\pi^{-1}(U)}$ is an automorphism. Assume that $(A\cdot \theta_*)=1.$

For every $b\in B$, let $X_b:=\pi^{-1}(b)$, $C_b:=C\cap X_b$, $f_b$ be the restriction of $f$ to $X_b$ and $A_b$ be the restriction of $L_B$ to $X_b.$ After shrinking $U$, we may assume that $C_b$ is irreducible for every $b\in U.$  For every $n\geq 0$ and $b\in U$, we have $((f^n)^*A\cdot A)=((f_b^n)^*A_b\cdot A_b)$. So $\la_1(f_b)=\la_1(f)=\la>1.$
For $b\in U$, set $$\theta_{*,b}':=\lim_{n\to \infty}((f_b^{-n})^*A_b\cdot A_b)/\la^n.$$ The discussion in Section \ref{subsectionnumgeo} shows that $\R\theta_{*,b}'$ the eigenspace of $(f_b^{-1})^*$ in $N^1(X_b)$  for eigenvalue $\la$. Set $\theta_{*,b}:=\theta_{*,b}'/(\theta_{*,b}'\cdot A).$
We have $$(\theta_{*,b}\cdot C_b)=(\theta_{*}\cdot C)>0.$$

\medskip

Set $I:=\{i\geq 0|\,\, f^i(x)\in C\}.$ For every $i\geq I$, $h^-(f^i(x))=\la^{-i}h^-(x)\leq h^-(x).$ By Lemma \ref{lemcomheigonc}, there is $M>0$ such that 
$h_A(f^i(x))<M$ for every $i\in I.$

For every point $y\in X$ defined over $K(B)$, its closure $s_y$ in $\sX$ is a section of $\pi.$ 
We may assume that for every $y\in X(K(B))$, $h_A(y)=(A_B\cdot s_y)$.
Also, for every section $s$ of $\pi$, its generic fiber defines a point $y_s\in X(K(B)).$
For every $y\in X(K(B))$, $\pi$ induces an isomorphism from $s_y$ to the curve
B. Consider the Hilbert polynomial
$$\chi(s^*_y\sO(nA_B))=1-g(B)+n(s_y\cdot A_B)=1-g(B)+nh_A(y).$$
So there is a quasi-projective $K$-variety $\sM_M$ that parameterizes the sections $s$ of $\pi$ with $h_A(y_s)\leq M$ (see \cite{debarre}).
For every $b\in U$, denote by $e_b: \sM_M\to X_b$ the morphism $s\mapsto s(b).$
Pick a sequence $b_i, i\geq 1$ of distinct points in $U(K)$. 
For  $s_1, s_2\in \sM_M$,  $s_1=s_2$ if and only if $e_{b_i}(s_1)=e_{b_i}(s_2)$ for every $i\geq 1.$
For $l\geq 1$, set 
$$e_l:=\prod_{i=1}^le_{b_i}: \sM_M\to \prod_{i=1}^lX_{b_i}.$$
By \cite[Lemma 8.1]{Xie2014}, there is $L\geq 1$ such that $e_L$ is quasi-finite.
For $j\in I$, $f^j(x)$ defines a point $s_{f^j(x)}\in \sM_M.$
The induction hypothesis shows that, for $i=1,\dots,L$,
$$e_{b_i}(\{f^j(x)|\,\, j\in I\})=\{f_{b_i}^j(x_{b_i})|\,\, j\in I\}\subseteq O_{f_{b_i}}(x_{b_i})\cap C_{b_i}$$ is finite.
So $e_L(\{f^j(x)|\,\, j\in I\})$ is finite. Since $e_L$ is quasi-finite, $O_f(x)\cap C=\{f^j(x)|\,\, j\in I\}$ is finite.
\endproof

\subsection{Conclusion}
Let $x\in X(\bk)$ and $C$ be an irreducible curve of $X$. If $(C\cdot \theta_*)>0,$ we conclude the proof by Lemma \ref{lemcthqezerfinite}.

Now assume that $(C\cdot \theta_*)=0.$ 
Let $B(f)$ be the set of curves $C'$ with $(C'\cdot \theta_*)=0.$
By Corollary \ref{corcurudst}, $C'\in B(f)$ if and only if  $(C'\cdot \theta^*)=0$, if and only if 
$(C'\cdot (\theta^*+\theta_*))=0$.  Since $\theta^*+\theta_*$ is big and nef, $B(f)$ is finite. 
Since $f^*\theta^*=\la_1\theta^*,$ $C'\in B(f)$ if and only if $f(C')\in B(f).$ So every curve in $B(f)$ is periodic. 
Since $C\in B(f),$ $C$ is periodic. 
$\square$

\section{Zariski dense orbit conjecture}
Let $X$ be a variety over $\bk$ of dimension $d_X.$
 Let $f: X\dashrightarrow X$ be a dominant rational self-map.

\subsection{Existence of well-defined orbits}\label{subsectionexweldef}
In characteristic $0$, the following result is well know. In positive characteristic, the proof is similar.
\begin{pro}\label{proxfnonempty}For every Zariski dense open subset $U$ of $X$, there is $x\in U(\bk)$ whose $f$-orbit is well defined and contained in $U.$
\end{pro}
\proof[Proof of Proposition \ref{proxfnonempty}]
After replacing $X, f$ by $U, f|_U$, we may assume that $X=U.$
So we only need to show that $X_f(\bk)\neq \emptyset.$

Let $\F$ be the smallest algebraically closed subfield of $\bk.$
So $\F=\overline{\Q}$ or $\overline{\F_p}.$
We may replace $\bk$ by an algebraically closed subfield $\bk'$ of $\bk$ with $\trd_{\F}\bk'<\infty$ such that $X,f$ are defined over $\bk'$.
Now assume that $\trd_{\F}\bk<\infty.$
If $\Char\,\bk=0$, we conclude the proof by \cite[Proposition 3.22]{Xie2019}. Now assume that $\Char\,\bk=p>0.$ 

The case $\bk=\overline{\F_p}$ is essentially proved in \cite[Proposition 5.5]{fa}. On may also see \cite[Proposition 6.2]{Xie2015}. In \cite[Proposition 6.2]{Xie2015}, $f$ is assumed to be birational, but its proof works for arbitrary dominant rational self-map.

Now assume that $\trd_{\F}\bk\geq 1.$ 
There is a subfield $L$ of $K$ which is finitely generated over $\bk$ such that $X,f$ are defined over $L.$
Let $B$ be a projective and normal variety over $\F$ such that $L=\bk.$ 
There is a $B$-scheme $\pi: X_B\to B$ and a rational self-map $f_B: X_B\dashrightarrow X_B$ over $B$ such that the geometric generic fiber of $(X_B,f_B)$ is $(X,f).$
Let $b$ be a general point of $B(\F)$ and denote by $(X_b,f_b)$ the fiber of $(X_B,f_B)$ above $b$. Then $V_b:=X_b\setminus  I(f_B)$ and $f_b$ is dominant.
Applying the case over $\overline{\F_p}$ to $(V_b, f_b|_{V_b})$, there is $x_b\in (V_b)_{f_b|_{V_b}}(\F).$
Cutting by general hyperplanes of $X_B$, there is an irreducible subvariety $S$ of $X_B$ of dimension $\dim S=\dim B$ passing through $b$ with $\pi(S)=B.$ 
Then the generic point of $S$ defines a point $x\in X_f(\bk)$, which concludes the proof. 
\endproof

\subsection{Tautological upper bound}

The following lemmas was proved in characteristic zero, but their proof works in any characteristic.
\begin{lemma}\cite[Lemma 2.15]{Jia2021}\label{l_inv_fun_field}
Let $K$ be an algebraically closed field extension of $\bk$.
Then $\bk(X)^{f}= \bk$ if and only if,  $K(X_{K})^{f_{K}}= K$.
\end{lemma}

\begin{lem}\cite[Lemma 2.1]{Xie2019}\label{leminvratfunite}Let $X'$ be an irreducible variety over $\bk$, $f': X'\dashrightarrow X'$ be a rational self-map and $\pi: X'\dashrightarrow X$ be a generically finite dominant rational map satisfying $f\circ \pi=\pi\circ f',$ then we have the following properties.
\begin{points}
\item If there exists $m\geq 1$, and $H\in \bk(X)^{f^m}\setminus \bk$, then there exists $G\in \bk(X)^f\setminus \bk$.
\item  There exists $H'\in \bk(X')^{f'}\setminus \bk$, if and only if there exists $H\in \bk(X)^f\setminus \bk$.
\end{points}
\end{lem}

They show that the assumption $\bk(X)^f=\bk$ is stable under base change, under positive iterate and under semiconjugacy by generaically finite dominant morphism. 
As an example of realization problems, the author asked the following question in \cite[Section 1.6]{Xie2019}.
\begin{que}\label{querealzdo} What is the minimal transcendence degree $R(\bk, X, f)$ of an algebraically closed 
field extension $K$ of $\bk$ such that $(X_K,f_K)$  satisfies the  ZDO property?
\end{que}

Proposition \ref{protautzdo} gives a tautological upper bound of $R(\bk, X, f)$.

\proof[Proof of Proposition \ref{protautzdo}]
We may assume that $\bk(X)^f= \bk$. By Lemma \ref{l_inv_fun_field}, $K(X_{K})^{f_{K}}= K$.

An irreducible $f_K$-invariant variety $V$ is said to be maximal, if the only irreducible $f_K$-invariant variety $W$ containing $V$ is $X_K.$
We note that $I(f_K)=I(f)\otimes_{\bk}K$ is defined over $\bk.$

\begin{lem}\label{leminvaroverk}Let $V$ be an irreducible $f_K$-invariant variety. Then $V$ is over defined over $\bk.$
\end{lem}
\proof[Proof of Lemma \ref{leminvaroverk}] 
Set $r:=\dim V<d_X.$
There is a subfield $L$ of $K$ which is finitely generated over $\bk$ such that $V$ is defined over $L.$
Let $B$ be a projective and normal variety over $\bk$ such that $L=\bk(B).$ 


Then there is a subvariety $V_B$ of $X\times B$ such that $\pi_2(V_B)=B$ where $\pi_2: X\times B\to B$ is the projection to the second coordinate and 
$V=V_{\eta}\times_{L}K$ where $\eta$ is the generic point of $B$ and $V_{\eta}$ is the generic fiber of $\pi_2|_{V_B}.$
We have $\dim V_B=\dim B+r.$
Since $V$ is $f_K$-invariant, $V_B\subseteq X\times B$ is $f_B:=f\times \id$ invariant.


Consider $\pi_1: X\times B\to B$ the projection to the first coordinate. 
It is clear that $\pi_1(V)$ is irreducible and $f$-invariant. 
Since $V\subseteq \pi_1(V)_K$ and $V$ is maximal, we get either $V_B=\pi_1^{-1}(\pi_1(V_B))$ or $\pi_2(V_B)=X.$
In the former case $V=\pi_1(V_B)_{K}$ is defined over $\bk.$ Now we assume that 
$\pi_2(V_B)=X.$ Then $\dim B= \dim V_B-r\geq d_X-r\geq 1$ and $\bk\subsetneq \pi_2^*(\bk(B))\subseteq \bk(V_B)^{f_B|_{V_B}}.$

If $\dim V_B=d_X$, we conclude the proof by Lemma \ref{leminvratfunite}. 
Now assume that $\dim V_B\geq d_X+1.$ 
So a general fiber of $\pi_1|_{V_B}$ has dimension $s\geq 1$. We have $\dim B=d_X+s-r>s.$
Let $H_1, \dots, H_2$ be very ample divisors on $B$ which are general in their linear system. Then the intersection of $\pi_2^{-1}(H_i), i=1,\dots,s$ and a general fiber of $\pi_1|_{V_B}$ is of dimension $0$ and $W':=V_B\cap H_1\dots \cap H_s$ is $f_B$-invariant.
Because $\pi_1(W')=X,$ there is an irreducible component $W$ of $W'$ with $\pi_1(W)=X$ and there is $l\geq 1$ such that $W$ is $f^l$-invariant. 
Because $d_X=\dim W$ and $\dim\pi_2(W)=d_X-r>0.$ So $\bk\subsetneq\bk(W)^{(f_B|_W)^l},$ which is a contradiction by Lemma \ref{leminvratfunite}.
\endproof

We only need to treat the case $\trd_{\bk}K=d.$ So we may assume that $K=\overline{\bk(X)}.$  The diagonal $\Delta$ of $X\times X$ defines a point $o$ in $X_K(K).$ Here we view $X_K$ as the geometric generic fiber of the second projection $\pi_2: X\times X\to X.$
Because $\pi_1(\Delta)=X$ where $\pi_1: X\times X\to X$ is the first projection, $O_{f_K}(o)$ is well defined and for every $n\geq 0$, $f_K^n(o)$ is not contained in any proper subvariety of $X_K$ defined over $\bk.$ An irreducible component $W$ of $\overline{O_{f_K}(o)}$ of maximal dimension is $f_K$-periodic and does not contained in any proper subvariety of $X_K$ defined over $\bk.$ By Lemma \ref{leminvaroverk}, $W=X_K$ which concludes the proof.
\endproof

In fact, with a slight modification, we prove a stronger result related to the strong form of the Zariski dense orbit conjecture \cite[Conjecture 1.4]{Xie2019}.

\begin{pro}\label{protautstzdo}Assume that $\bk(X)^f=\bk.$
Let $K$ be an algebraically closed field extension of $\bk$ with $\trd_{\bk}K\geq \dim X$.
Then for every nonempty Zariski open subset $U$ of $X_K$, there is a point $x\in U(K)$ whose $f_K$-orbit is well defined and contained in $U.$
\end{pro}
\proof[Proof of Proposition \ref{protautstzdo}]Keep the notation in the proof of  Proposition \ref{protautzdo}.
Pick a general point $b\in X(\bk).$ Then $U_b:=X\setminus\overline{(X_K\cap U)}$ is not empty. 
By Proposition \ref{proxfnonempty}, there is $x_b\in U_b$, whose $f$ orbit is well defined and contained in $U_b.$
Cutting by general hyperplanes of $X\times X$, there is an irreducible subvariety $S$ of $X\times X$ of dimension $\dim S=\dim X$ passing through $(x_b,b)$ such that 
$\pi_1(S)=X$ and $\pi_2(S)=X.$ The generic point of $S$ defines a point in $x\in X_K(K).$ Then the $f_K$-orbit of $x$ is well defined and contained in $U.$
After replacing $o$ by $x$, the argument in the last paragraph of the proof of Proposition \ref{protautzdo} shows that $O_{f_K}(x)$ is Zariski dense in $X_K.$
\endproof

\subsection{Height argument}
The aim of this section is to prove Theorem \ref{thmendononpre}, \ref{thmzdosuraut} and \ref{thmzdola}.

Assume that $\Char\, \bk=p>0$ and $\trd_{\overline{F_p}} \bk\geq 1.$ Let $f: X\to X$ be a dominant  endomorphism of a projective variety.
There is a algebraically closed subfield $K$ of $\bk$ such that $\trd_K\bk=1.$
So there is smooth projective curve $B$ over $K$, such that $f,X$ are defined over $K(B).$  The Weil heights appeared in the section are associated to the function field $K(B).$

\proof[Proof of Theorem \ref{thmendononpre}] By Corollary \ref{cordenseal}, there exists a point $x\in U(\bk)$ with $\alpha_f(x)=\la_1(f)>1$ and $O_f(x)\subseteq U$. So $x$ has infinite orbit.
\endproof

\proof[Proof of Theorem \ref{thmzdola}]The proof of \cite[Proposition 8.6]{Jia2021} shows that 
for every $f$-periodic proper subvariety $V$ of period $m\geq 1,$ $\la_1(f^m|_V)<\la_1(f^m).$
By Propositon \ref{proextarhitd}, there exists a point $x\in X(\bk)$ with $\alpha_f(x)=\la_1(f)>1$.
Let $W$ be an irreducible component of $\overline{O_f(x)}$ of maximal dimension. There is $m\geq 1$ with $f^m(W)=W.$
There is $l\geq 0$ such that $f^l(x)\in W$. 

If $W\neq X$,
by Proposition \ref{proupboundarth} and Lemma \ref{lem_subvar}, we get 
$$\la_1(f)^m=\overline{\alpha}_{f}(x)^m=\overline{\alpha}_{f^m}(f^l(x))\leq \la_1(f^m|_W)<\la_1(f)^m.$$
We get a contradiction. So $W=X$, which concludes the proof.
\endproof

The following theorem was proved in \cite[Theorem 1]{Invarianthypersurfaces}, but when $f$ is an automorphism,  its proof work in arbitrary characteristic.
\begin{thm}\label{thminvhyper}If $f$ is an automorphism and it preserves infinitely many (not necessarily irreducible) hyperplanes, then $\bk(X)^f\neq \bk.$
\end{thm}

\begin{pro}\label{proautbounded}Let $X$ be a projective variety over $\bk$ of dimension $d_X$. Let $L$ be an ample line bundle on $X$.
Let $f: X\to X$ be an automorphism such that $((f^n)^*L\cdot L^{d_X-1}),n\geq 0$ is bounded. Then $(X,f)$ satisfies the ZDO property.
\end{pro}

\proof[Proof of Proposition \ref{proautbounded}]
Let $\Aut(X)$ be the scheme of automorphisms of $X.$
Every connected component of $\Aut(X)$ is a variety over $\bk$, but $\Aut(X)$ may have infinite connected component.

Because $((f^n)^*L\cdot L^{d_X-1}),n\geq 0$ is bounded, the Zariski closure $G$ of $f^n, n\geq 0$ in $\Aut(X)$ is a commutative algebraic group.
After replacing $f$ by a suitable iterate, we may assume that $G$ is irreducible. We may assume that $f$ is of infinite order. So $\dim G\geq 1.$

For every $x\in X(\bk)$, $\overline{O_f(x)}=\overline{G.x}$.
Consider the morphism $\Phi: G\times X\to X\times X$ sending $(g,x)$ to $(g(x),x)$.
Denote by $\pi_i: X\times X\to X$ the $i$-th projection. 
Consider the $G$-action on $X\times X$ by $g.(x,y)=(g(x),y).$
Set $F:=f\times \id: X\times X\to X\times X.$ 

The image $W$ of $\Phi$ is a constructible subset of $X\times X$.
Let $Y$ be the Zariski closure of $W$ in $X\times X$. It is irreducible and $F$-invariant.
Let $\Delta$ be the diagonal of $X\times X.$ Then $\Delta\subseteq W\subseteq Y.$
So $\pi_1(Y)=\pi_2(Y)=X.$ Because $\dim G\geq 1$ and the action of $G$ on $X$ is faithful, $Y\neq \Delta.$
So the general fiber of $\pi_2|_{Y}$ has dimension $r\geq 1.$ If $r=\dim X$, then for a general $x\in X(\bk)$, $\overline{O_f(x)}=\overline{G.x}=X$ which concludes the proof.
Now assume that $r<\dim X.$

We have $\dim Y=\dim X+r.$
The general fiber of $\pi_1|_{Y}$ also has dimension $r\geq 1.$
Let $H_1,\dots,H_r$ be very ample hyperplanes of $X$ which are general in their linear system.
The intersection of $\pi_2^*H_1,\dots, \pi_2^*H_r$ and a general fiber of $\pi_1|_{Y}$ is proper. 
Set $Z:=\pi_2^{-1}(\cap_{i=1}^r H_i).$ We have $\pi_1(Z)=X$, $\dim Z=\dim X$ and $\dim\pi_2(Z)=\dim(H_1\cap \dots \cap H_r)=\dim X-r\geq 1.$
Because $G$ is connected, every irreducible component of $Z$ is $G$-invariant. 
In particular, let $T$ be an irreducible component of $Z$ with $\pi_1(T)=X$, then $T$ is $F$ invariant and we have 
 $\dim T=\dim X$, $\dim\pi_2(T)=\dim X-r\geq 1.$ Because $\bk\subsetneq \bk(T)^{F|_T}$ and $\pi_1\circ F|_T=f\circ \pi_2,$ we conclude the proof by Lemma \ref{leminvratfunite}.
\endproof

\begin{thm}\label{thmzdosuraut}Assume that $\Char\, \bk=p>0$ and $\trd_{\overline{F_p}} \bk\geq 1.$ 
Let $f: X\to X$ be an automorphism of a projective surface. Then $(X,f)$ satisfies the ZDO property.
\end{thm}

\proof[Proof of Theorem \ref{thmzdosuraut}]
By \cite{Lipman1978}, there is a minimal desingularization $\pi:X'\to X$. Then one may lift $f$ to an automorphism $f'$ of $X'.$
Easy to see that $(X,f)$ satisfies the ZDO property if and only if $(X',f')$ satisfies the ZDO property.
After replacing $(X,f)$ by $(X',f')$, we may assume that $X$ is smooth.
By Theorem \ref{thmzdola}, we may assume that $\la_1(f)=1.$ 
Let $L$ be an ample line bundle on $X$.

If $((f^n)^*L\cdot L), n\geq 0$ is unbounded, by Gizatullin \cite{Gizatullin1980}, there is a surjective morphism $\pi: X\dashrightarrow C$ to a smooth projective curve $C$ and an automorphism $f_C: C\to C$ such that $f_C\circ \pi=\pi\circ f.$
\footnote{In \cite{Gizatullin1980}, there is an assumption that $\Char\, \bk\neq 2,3.$ 
But, it is checked in \cite{Cantat2019a} that such assumption in  \cite{Gizatullin1980} can be removed.}
After replacing $\pi:X\dashrightarrow C$ by a minimal resolution of $\pi$, we may assume that $\pi$ is a morphism.
There is $m\geq 1$ such that $f_C^m=\id$, we have $\bk\subsetneq \pi^*(\bk(C)^{f_C})\subseteq \bk(X)^f$.

Now we may assume that $((f^n)^*L\cdot L), n\geq 0$ is bounded. We conclude the proof by Proposition \ref{proautbounded}.
\endproof

\section{Ergodic theory}\label{sectionergodictheory}
Let $X$ be a variety over $\bk$. 
Denote by $|X|$ the underling set of $X$ with the constructible topology i.e. the topology on a  $X$ generated by the constructible subsets.
This topology is finer than the Zariski topology on $X.$ 
Moreover $|X|$ is (Hausdorff) compact. Denote by $\eta$ the generic point of $X$.

 \medskip
 
 Using the Zariski topology, on may define a partial ordering on $|X|$ by $x\geq y$ if and only if $y\in \overline{x}.$
The noetherianity of $X$ implies that this partial ordering satisfies the descending chain condition: for every chain in $|X|$, 
$$x_1\geq x_2\geq \dots$$
 there is $N\geq 1$ such that $x_n=x_N$ for every $n\geq N.$
 For every $x\in |X|$, the Zariski closure of $x$ in $X$ is 
 $U_x:=\overline{\{x\}}=\{y\in |X||\,\, y\leq x\}$ which is open and closed in $|X|.$
 
 \medskip
 
Let $\sM(X)$ be the space of Radon measure on $X$ endowed with the weak-$\ast$ topology
and $\sM^1(|X|)$ be the space of probability Radon measure on $|X|.$ Note that $\sM^1(|X|)$ is compact.
 
 \proof[Proof of Theorem \ref{thmRadon}]
We claim that for every Radon measure $\mu$ on $|X|$ with $\mu(|X|)>0$, there exists $x\in X$ such that $\mu(x)>0.$

\medskip

Then for every Radon measure $\mu$ on $|X|$, set $S(\mu):=\{x\in |X||\,\, \mu(x)>0\}.$ Then $S(\mu)$ is at most countable and we have $c:=\sum_{x\in S(\mu)}\mu(x)\in (0,\mu(|X|)].$
If $c=\mu(|X|)$, then we have $\mu=\sum_{x\in S(\mu)}\mu(x)\delta_x$, which concludes the proof.
Assume that $c<\mu(|X|)$, set $$\alpha:=\mu-\sum_{x\in S(\mu)}\mu(x)\delta_x.$$ Then $\alpha$ is a Radon measure with $\alpha(|X|)=\mu(|X|)-c>0$ and $S(\alpha)=\emptyset$. This contradicts our claim.

\medskip

Now we only need to prove the claim.
\begin{lem}\label{lemfindnexp}For $x\in |X|$, if $\mu(U_x)>0$ and  $\mu(x)=0$, then there exists $y\in U_x\setminus \{x\}$ such that $\mu(U_y)>0.$
\end{lem}
Now assume that for every $x\in |X|$, $\mu(x)=0.$
Since $|X|=\cup_{x\in X}U_x$ and $|X|$ is compact, there exists a finite subset $F$ of $|X|$ such that $|X|=\cup_{x\in F}U_x.$
Then there exists $x_0\in F$ such that $\mu(U_{x_0})>0.$
Since $\mu(x_0)=0$ by the assumption,
by Lemma \ref{lemfindnexp}, we get a sequence of points $x_i, i\geq 0$, $x_i> x_{i+1}$ such that $\mu(U_{x_i})>0, \mu(x_i)=0.$ 
This contradicts the descending chain condition. 
\endproof

\proof[Proof of Lemma \ref{lemfindnexp}]
Observe that $U_x\setminus \{x\}$ is open and $\mu(U_x\setminus \{x\})>0.$
Since $\mu$ is Radon, there exists a compact subset $K\subseteq U_x\setminus \{x\}$ such that $\mu(K)>0.$
Since $K\subseteq \cup_{z\in K}U_z$, there exists a finite set $x_1,\dots,x_m$ in $K$ such that 
$K\subseteq \cup_{i=1}^mU_{x_i}.$ Since $\sum_{i=1}^m\mu(U_{x_i})\geq \mu(K)>0,$ there exists some $1\leq i\leq m$ such that $\mu(U_{x_i})>0.$
Set $y:=x_i$, we concludes the proof.
\endproof

\proof[Proof of Corollary \ref{corgenericseqence}]
Let $x_n\in X, n\geq 0$ be a sequence of points.

We first assume that $x_n\in X, n\geq 0$ is generic.
Because $\sM^1(|X|)$ is compact, we only need to show that for every subsequence with
$\lim_{i\to \infty}\delta_{x_{n_i}}=\mu$, we have $\mu=\delta_{\eta}.$
By Theorem \ref{thmRadon}, we may write 
$$\mu=\sum_{i\geq  0}^ma_i\delta_{x_i}$$ where $m\in \Z_{\geq 0}\cup \{\infty\}$, $x_i$ are distinct points, $a_i> 0$ and $\sum_{i\geq 0}a_i=1.$
If $\mu\neq \delta_{\eta},$ we may assume that $x_0\neq \eta.$ Then $V:=\overline{\{x_0\}}$ is a closed proper subvariety of $X.$
Then we have $$1_{V}(x_{n_i})=\int 1_V\delta_{x_{n_i}}\to \int 1_V\mu>a_0$$
as $n\to \infty.$ So $x_{n_i}\in V$ for all but finitely many $i$, which is a contradiction.

Now assume that $\lim_{n\to \infty}\delta_{x_{n}}=\delta_{\eta}.$ For every subsequence $x_{n_i}, i\geq 0$ and every closed proper subvariety $V$ of $X,$
$$\lim_{i\to\infty}1_V(x_{n_i})=\lim_{i\to\infty}\int1_V\delta_{x_{n_i}}=\int 1_V\delta_{\eta}=0.$$
So $x_{n_i}\not\in V$ for all but finitely many $i$. So $x_{n_i}$ is Zariski dense in $X$.
\endproof

\subsection{DML problems}
 Let $f: X\dashrightarrow X$ be a dominant  rational self-map. Set $|X|_f:=|X|\setminus (\cup_{i\geq 1}I(f^i)).$
Because every Zariski closed subset of $X$ is open and closed in the constructible topology,  $|X|_f$ is a closed subset of $|X|.$
The restriction of $f$ to $|X|_f$ is continuous. We still denote by $f$ this restriction.

 \medskip

  Denote by $\sP(X,f)$ the set of $f$-periodic points in $|X|_f.$
 Theorem \ref{thmRadon} implies directly the following lemma.

 \begin{lem}\label{leminvm}If $\mu\in \sM^1(|X|_f)$ with $f_*\mu=\mu$, then there are $x_i\in \sP(X,f), i\geq 0$ and $a_i\geq 0, i\geq 0$ with $\sum_{i=0}a_i=1$
 such that  
 $$\mu=\sum_{i\geq 0}\frac{a_i}{\#O_f(y)}(\sum_{y\in O_f(x_i)}\delta_y)$$
 \end{lem}
 
 Now we prove Theorem \ref{thmdml} and Theorem \ref{thmdmlback}.
 \proof[Proof of Theorem \ref{thmdml}]
 Let  $x$ be a points $\in X_f(\bk)$ with $\overline{O_f(x)}=X.$ Let $V$ be a proper subvariety of $X$. 
 Consider a sequence of intervals $I_n, n\geq 0$ in $\Z_{\geq 0}$ with $\lim\limits_{n\to \infty}\# I_n=+\infty$.
 For every $n\geq 0$, set $\mu_n:=(\#I_n)^{-1}(\sum_{i\in I_n}\delta_{f^i(x)})\in \sM^1(|X|_f).$
 Because $$\frac{\#(\{n\geq 0|\,\, f^n(x)\in V\}\cap I_n)}{\#I_n}=\int 1_V\mu_n,$$
 we only need to show that \begin{equation}\label{equlimmundml}\lim_{n\to \infty}\mu_n=\delta_{\eta}.\end{equation}
 Because $\sM^1(|X|)$ is compact, we only need to show that for every convergence subsequence $\mu_{n_i}, i\geq 0$,
 $\mu_{n_i}\to \delta_{\eta}$ as $i\to \infty.$
 Set $\mu:=\lim_{n\to \infty}\mu_{n_i}.$ We have 
 $$f_*\mu=\lim_{n\to \infty}f_*\mu_{n_i}=\lim_{i\to \infty}\mu_{n_i}+\lim_{i\to \infty}(\#I_{n_i})^{-1}(\delta_{f^{\max I_{n_i}+1}(x)}-\delta_{f^{\min I_{n_i}}(x)})$$
$$=\lim_{i\to \infty}\mu_{n_i}=\mu.$$
For every $y\in \sP(X,f)\setminus\{\eta\}$, $U_y$ is open and closed in $|X|_f.$
Then $$Y:=|X|_f\setminus (\cup_{y\in \sP(X,f)}U_y)$$ is an $f$-invariant closed proper subset of $|X|_f.$
Because $\overline{O_f(x)}=X,$ $x\in Y$. So for every $n\geq 0$, $\Supp\, \mu_n\subseteq Y.$
Because $Y\cap \sP(X,f)=\{\eta\},$
 Lemma \ref{leminvm} shows that $\mu=\delta_{\eta}.$
  \endproof
 
 \proof[Proof of Theorem \ref{thmdmlback}]
 Let $x_n \in X_f(\bk), n\leq 0$ be a sequence of points such that $\overline{\{x_n, n\leq 0\}}=X$ and $f(x_n)=x_{n+1}$ for all $n\leq -1.$
 Consider a sequence of intervals $I_n, n\geq 0$ in $\Z_{\leq 0}$ with $\lim\limits_{n\to \infty}\# I_n=+\infty$.
 For $n\geq 1,$ define $x_n:=f^n(x_0).$
 
For every $n\geq 0$, set $\mu_n:=(\#I_n)^{-1}(\sum_{i\in I_n}\delta_{x_i})\in \sM^1(|X|_f).$
 As the proof of Theorem \ref{thmdml}, we only need to show 
 \begin{equation}\label{equlimmunbackdml}\lim_{n\to \infty}\mu_n=\delta_{\eta}.\end{equation}
 Because $\sM^1(|X|)$ is compact, we only need to show that for every convergence subsequence $\mu_{n_i}, i\geq 0$,
 $\mu_{n_i}\to \delta_{\eta}$ as $i\to \infty.$  Set $\mu:=\lim_{n\to \infty}\mu_{n_i}.$ We have 
 $$f_*\mu=\lim_{n\to \infty}f_*\mu_{n_i}=\lim_{i\to \infty}\mu_{n_i}+\lim_{i\to \infty}(\#I_{n_i})^{-1}(\delta_{x_{\max I_{n+1}+1}}-\delta_{x_{\min I_{n+1}+1}})$$
$$=\lim_{i\to \infty}\mu_{n_i}=\mu.$$
 For every $y\in \sP(X,f)\setminus \{\eta\},$ $U_y\cap \{x_i, i\leq 0\}$ is finite.
 Otherwise $\{x_i, i\leq 0\}\subseteq \cup_{z\in O_f(y)}U_z$ is not Zariski dense in $X.$
 This implies that $\mu(U_y)=\lim_{i\to \infty}\mu_{n_i}(U_y)=0.$ So $\Supp\, \mu\subseteq Y:=|X|_f\setminus (\cup_{y\in \sP(X,f)}U_y).$
 Because $Y\cap \sP(X,f)=\{\eta\},$
 Lemma \ref{leminvm} shows that $\mu=\delta_{\eta}.$
 \endproof
 
\subsection{Functoriality}\label{subsecfunct}
Assume that $f: X\to X$ is a flat and finite endomorphism.  
Because the image by $f$ of every constructible subset is constructible, $f$ is open w.r.t the constructible topology.
Moreover, for every $x\in X$, $f(U_x)=U_{f(x)}.$

\medskip

Denote by $C(|X|)$ the space of continuous $\R$-valued functions on $|X|$ with the $L_{\infty}$ norm $\|\cdot\|.$
For every $\phi\in C(|X|)$, define $f_*\phi$ to be the function
$$x\in |X| \mapsto f_*\phi:=\sum_{y\in f^{-1}(x)}m_f(y)\phi(y).$$

The following Lemma shows that $f_*$ is a bounded linear operator on $C(|X|).$
\begin{lem}\label{lempushffun}For every $\phi\in C(|X|)$, $f_*\phi$ is continuous and $\|f_*\phi\|\leq d_f\|\phi\|.$
\end{lem}
\proof 
By \cite[Proposition 2.8]{Gignac2014a}, for every $x\in |X|$, there is an open subset $V_x\subseteq U_x$ containing $x$ such that 
$V_x=f^{-1}(f(V_x))\cap U_x$ and for every $y\in f(V_x)$, 
$$m_f(x)=\sum_{z\in f^{-1}(y)\cap V_x}m_f(z).$$
Because $\{x\}=f^{-1}(f(x))\cap U_x$, such $V_x$ can be taken arbritarily small.

\medskip

Because $\phi\in C(|X|),$ for every $x\in |X|$ and $r>0$, there is an open subset $V_x^r$ containing $x$ such that for every $y\in V_x^r$,
$|\phi(y)-\phi(x)|<r.$

Let $w$ be a point in $|X|$. There are open neighborhoods $O_y$ of $y\in f^{-1}(w)$,   such that for distinct $y_1,y_2\in f^{-1}(w),$
$O_{y_1} \cap O_{y_2}=\emptyset.$
For every $r>0$, and $y\in f^{-1}(w)$, we may take $V_y$ as in the first paragraph such that 
$V_y\subseteq O_y\cap V^{r/d_f}_y.$
Then $W^r_w:=\cap_{y\in f^{-1}(w)}f(V_y)$ is an open set containing $w.$
For every $x\in W^r_w$ and distinct $y_1,y_2\in  f^{-1}(w)$, we have 
$$(f^{-1}(x)\cap V_{y_1})\cap (f^{-1}(x)\cap V_{y_2})=\emptyset.$$
Since $$d_f=\sum_{z\in f^{-1}(x)}m_f(z)\geq \sum_{y\in f^{-1}(w)}\sum_{z\in f^{-1}(x)\cap V_{y}}m_f(x)=\sum_{y\in f^{-1}(w)}m_f(y)=d_f,$$
we have $$f^{-1}(x)=\sqcup_{y\in f^{-1}(w)} (f^{-1}(x)\cap V_{y}).$$
Then we get 
$$|f_*\phi(x)-f_*\phi(w)|\leq \sum_{y\in f^{-1}(w)}|m_f(y)\phi(y)-\sum_{z\in V_y\cap f^{-1}(x)}m_f(z)\phi(z)|$$
$$\leq \sum_{y\in f^{-1}(w)}\sum_{z\in V_y\cap f^{-1}(x)}m_f(z)|\phi(y)-\phi(z)|< \sum_{y\in f^{-1}(w)}\sum_{z\in V_y\cap f^{-1}(x)}m_f(z)r/d_f=r.$$
So $f_*\phi$ is continuous. Moreover for every $x\in |X|$
$$f_*\phi(x)=|\sum_{y\in f^{-1}(x)}m_f(x)\phi(y)|\leq \sum_{y\in f^{-1}(x)}m_f(x)\|\phi\|=d_f\|\phi\|,$$
which concludes the proof.\endproof

Now one may define the pullback $f^*:\sM(|X|)\to \sM(|X|)$ by the duality:
for every $\mu\in \sM(|X|)$ and $\phi\in C(|X|)$, $$\int \phi (f^*\mu)=\int (f_*\phi)\mu.$$
In particular, $f^*\mu(|X|)=d_f\mu(|X|).$
The pullback $f^*:\sM(|X|)\to \sM(|X|)$ is continuous w.r.t. the weak-$\ast$ topology on $\sM(|X|)$ and one may check that for every $x\in |X|,$
$$f^*\delta_x=\sum_{y\in f^{-1}(x)}m_f(y)\delta(y).$$

\subsection{Backward orbits}
Assume that $f: X\to X$ is a flat and finite endomorphism.  In particular, $f$ is surjective.
The aim of this section is to prove Theorem \ref{thmequpullback}, \ref{thmcountprestr} and \ref{thmseqd}.

\medskip

Let $TP(X,f)$ be the point $x\in |X|$ such that $\cup_{n\geq 0}f^{-n}(x)$ is finite.
It is clear that $f^*TP(X,f)\subseteq TP(X,f).$
For $x\in TP(X,f)$, since
$f:\cup_{n\geq 1}f^{-n}(x)\to \cup_{n\geq 0}f^{-n}(x)$ is surjective,
it is bijective. So $x$ is periodic. Then  $f^{-1}(TP(X,f))=TP(X,f)$ and for every $x\in TP(X,f)$, $f^{-1}(x)$ is a single point.
For the simplicity, we still denote by $f^{-1}(x)$ the unique points in it.

For every $x\in TP(X,f)$, $f^{-1}(U_x)=\cup_{y\in f^{-1}(x)}U_{f^{-1}(y)}.$
Then $$Y:=X\setminus \cup_{x\in TP(X,f)\setminus \{\eta\}}U_x$$ is a closed subset of $|X|$ such that $f^{-1}(Y)=f(Y)=Y.$
It is clear that $Y$ is exactly the subset of $x\in |X|$ such that $\overline{\cup_{i\geq 0}f^{-i}(x)}=X.$

\begin{lem}\label{lemtotinm}For $\mu\in \sM(|X|)$ supported in $Y$, if $d_f^{-1}f^*\mu=\mu$, then $\mu=\delta_{\eta}.$
\end{lem}
\proof
Assume that $\mu\neq \delta_{\eta}.$
We may assume that $\mu(\eta)=0$. Otherwise, we may replace $\mu$ by $\mu-\mu(\eta)\delta_{\eta}.$ 
By Theorem \ref{thmRadon}, one may write 
$$\mu=\sum_{i= 0}^ma_i\delta_{x_i}$$ where $m\in \Z_{\geq 0}\cup \{\infty\}$, $x_i$ are distinct points in $Y\setminus \{\eta\}$, $a_i> 0$ and $\sum_{i\geq 0}a_i=1.$
We have 
$$\mu=d_f^{-1}f^*\mu=\sum_{i=0}^m\sum_{y\in f^{-1}(x)}\frac{a_im_f(y)}{d_f} \delta_y.$$
Terms in the right hand side have distinct supports.

Assume that $a_i$ is decreasing. We claim that for every $i,$ $f^{-1}(x_i)$ is a single point. 
Otherwise, pick $l$ minimal such that $f^{-1}(x_l)$ is not a single point.
Assume that $s\geq 0$ is maximal such that $a_{l+s}=a_l.$
Think $\mu$ as a function $\mu: |X|\to [0,1]$ sending $x$ to $\mu(x)$. 
We have $\mu^{-1}(a_l)=s+1.$ On the other hand
$$(d_f^{-1}f^*\mu)^{-1}(a_l)=\{i=l,\dots,l+s|\,\, f^{-1}(x_i) \text{ is a single point}\}\leq s,$$
which is a contradiction.
Then we get $\mu=\sum_{i=0}^ma_i\delta_{f^{-1}(x_i)}.$
Because for every $r>0$, $\{i=0,\dots,m|\,\, a_i\geq r\}$ is finite,
all $x_i, i=0,\dots, m$ are contained in $TP(X,f)\cap (Y\setminus\{\eta\})=\emptyset.$ 
We get a contradiction.
\endproof

\medskip

\proof[Proof of Theorem \ref{thmequpullback}]
Let $x$ be a point in $X(\bk)$ with $\overline{\cup_{i\geq 0}f^{-i}(x)}=X.$
Let  $I_n, n\geq 0$ be a sequence of intervals in $\Z_{\geq 0}$ with $\lim_{n\to \infty}\# I_n=+\infty$.
Set $$\mu_n:=\frac{1}{\#I_n}(\sum_{i\in I_n}d_f^{-i}(f^i)^*\delta_{x})\in \sM^1(|X|).$$
Because $\sM^1(|X|)$ is compact, only need to show that for every convergence subsequence $\mu_{n_i}, i\geq 0$,
 $\mu_{n_i}\to \delta_{\eta}$ as $i\to \infty.$ Set $\mu:=\lim_{n\to \infty}\mu_{n_i}.$
 
 Then $$f^*\mu=\lim_{i\to \infty}f^*\mu_{n_i}=\lim_{i\to \infty}\frac{1}{\#I_n}(\sum_{j\in I_{n_i}}d_f^{-j}(f^{j+1})^*\delta_{x})$$
 $$\lim_{i\to \infty}d_f\mu_{n_i}+\lim_{i\to \infty}\frac{d_f}{\#I_n}(d_f^{-\max I_{n_i}-1}(f^{\max I_{n_i}+1})^*\delta_{x}-d_f^{-\min I_{n_i}}(f^{\min I_{n_i}})^*\delta_{x})$$
Because $d_f^{-\max I_{n_i}-1}(f^{\max I_{n_i}+1})^*\delta_{x}(|X|)=d_f^{-\min I_{n_i}}(f^{\min I_{n_i}})^*\delta_{x}(|X|)=1,$
we get $$f^*\mu=\lim_{i\to \infty}d_f\mu_{n_i}=d_f\mu.$$
 Because $x\in Y$, for every $n\geq 0$, $\Supp\, \mu_n\subseteq Y.$ So $\mu\subseteq Y.$
 We conclude the proof by Lemma \ref{lemtotinm}.
\endproof

\proof[Proof of Theorem \ref{thmcountprestr}]
Assume that $\bk(X)/f^*\bk(X)$ is separable. Let $x\in X(\bk)$ be a point with $\overline{\cup_{i\geq 0}f^{-i}(x)}=X.$
Pick $c\in (0,1]$
Because $$\#f^n(x)\leq \sum_{y\in f^{-n}(x)}m_{f^n}(y)=d_f^n,$$ we have 
$$\limsup_{n\to \infty} (S^n_c)^{1/n}\leq \limsup_{n\to \infty} \#f^n(x)^{1/n}\leq d_f.$$
We now prove the inequality in the other direction.

\medskip

By \cite[Theorem 2.1]{Gignac2014a} and \cite[Proposition 2.3]{Gignac2014a}, there is a proper Zariski closed subset $R$ of $X$, such that 
for every $y\in X(\bk)\setminus R,$ $m_f(y)=1.$ 
Set $$\mu_n:=\frac{1}{n}(\sum_{i=1}^{n}d_f^{-i}(f^i)^*\delta_{x})\in \sM^1(|X|).$$
By Theorem \ref{thmequpullback}, \begin{equation}\label{equmutoz}\lim_{n\to \infty}\mu_n=\delta_{\eta}.\end{equation}

\medskip

Set $D:=\{1,\dots, d_f\}$. Let $\Omega:=\sqcup_{n\geq 0}D^n$ be the set of words in $D$ of finite length. In particular $D^{0}=\{\emptyset\}.$
By induction, one may define a map $$\phi: \Omega\to \sqcup_{n\geq 0} f^{-n}(x)\subseteq \sqcup_{n\geq 0} X$$ such that 
\begin{points}
\item $\theta(D^n)=f^{-n}(x), $ in particular $\phi(\emptyset)=x.$
\item for every word $w_1\dots w_n\in D^n, n\geq 1,$  $$\theta(w_1\dots w_{n-1})=f(\theta(w_1\dots w_{n}));$$
\item for every $y\in f^{-n-1}(x)$ and $w_1\dots w_n\in D^n$ satisfying $\theta(w_1\dots w_{n})=f(y),$
$$\#\{w\in D|\,\, \theta(w_1\dots w_nw)=y\}=m_f(y).$$
\end{points}
By \cite[Proposition 2.5]{Gignac2014a}, for every $y\in f^{-n-1}(x)$, $m_{f^{n+1}}(y)=m_{f^n}(f(y))m_f(y).$
This implies that for every $y\in f^{-n}(x),$ $$\#\{\omega\in D^n|\,\, \theta(\omega)=y\}=m_{f^n}(y).$$
Define a function $A:\Omega\to (0,1]$ by $$A: \omega\in D^n\mapsto m_{f^n}(\theta(\omega))^{-1}.$$
We have 
\begin{points}
\item $\sum_{\omega\in D^n}A(\omega)=\#f^{-n}(x);$
\item for every $w_1\dots w_{n+1}\in D^{n+1}$, $$A(w_1\dots w_{n+1})=m_f(\theta(w_1\dots w_{n+1}))^{-1}A(w_1\dots w_{n}).$$
\end{points}
We have $A(\emptyset)=1$ and 
$$A(w_1\dots w_{n+1})\geq d_f^{-1_R(\theta(w_1\dots w_{n+1}))}A(w_1\dots w_{n}).$$
Then we have 
$$\prod_{\omega\in D^{n+1}}A(\omega)=\prod_{\omega\in D^{n}}\prod_{w\in D}A(\omega w)\geq \prod_{\omega\in D^{n}}\prod_{w\in D}d_f^{-1_R(\theta(w_1\dots w_{n+1}))}A(\omega)$$
$$=(\prod_{\omega\in D^{n+1}}d_f^{-1_R(\theta(\omega))})(\prod_{\omega\in D^{n}}A(\omega))^{d_f}=d_f^{-\int1_R(f^{n+1})^*\delta_x}(\prod_{\omega\in D^{n}}A(\omega))^{d_f}.$$
Set $B_n:=\log_{d_f}\prod_{\omega\in D^{n}}A(\omega).$
We get $$B_{n+1}/d_f^{n+1}\geq -d_f^{-n-1}\int1_R(f^{n+1})^*\delta_x+B_{n}/d_f^n.$$
Then we get 
$$B_n/d_f^{n}\geq \sum_{i=1}^{n}-d_f^{-i}\int1_R(f^{i})^*\delta_x=-n\int 1_R\mu_n.$$

For every $n\geq 0$, pick $E_n\subseteq f^{-n}(x)$, such that $$\sum_{y\in E_n}m_{f^n}(y)\geq cd_f^n$$ and $\#E_n=S^n_c.$
So $$\#\theta^{-1}(E_n)=\sum_{y\in E_n}m_{f^n}(y)\geq cd_f^n.$$

By Inequality of arithmetic and geometric means,  we have 
$$S^n_c=\sum_{\omega\in \theta^{-1}(E_n)}A(\omega)\geq \#\theta^{-1}(E_n)(\prod_{\omega\in \theta^{-1}(E_n)}A(\omega))^{\frac{1}{\#\theta^{-1}(E_n)}}$$
$$\geq cd_f^n(\prod_{\omega\in \theta^{-1}(E_n)}A(\omega))^{\frac{1}{cd_f^n}}\geq cd_f^n(\prod_{\omega\in D^n}A(\omega))^{\frac{1}{cd_f^n}}$$
$$=cd_f^{n+B_n/cd_f^{n}}\geq cd_f^{n(1-c^{-1}\int 1_R\mu_n)}.$$


So $(S^n_c)^{1/n}\geq c^{1/n}d_f^{1-\int 1_R\mu_n}.$ By Equality \ref{equmutoz},
$$\liminf_{n\geq 0}(S^n_c)^{1/n}\geq d_f,$$ whcih concludes the proof.
\endproof

\medskip

\proof[Proof of Theorem \ref{thmseqd}]
Set $d_X:=\dim X.$
Assume that $\bk(X)/f^*\bk(X)$ is separable and $$\la_{\dim X}(f)>\max_{1\leq i\leq \dim X-1} \la_i.$$
Let $x$ be a point in $X(\bk)$ with $\overline{\cup_{i\geq 0}f^{-i}(x)}=X.$

\medskip

We first show that for every irreducible subvariety $V$ of $X$ of $\dim V=d_V<d_X$, 
\begin{equation}\label{equintvnum}\limsup_{n\to \infty}\#(f^{-n}(x)\cap V)^{1/n}\leq \la_{d_V}.
\end{equation}
Let $Y$ be a normal and projective variety containing $X$ as an Zariski dense open subset. 
Let $Z$ be the Zariski closure of $V$ in $W.$ Let $\sI_Z$ be the ideal sheaf associated to $Z.$
Let $H$ be a very ample divisor on $Y$ such that 
$\sO(H)\otimes \sI_Z$ is generated by global sections.

For every $n\geq 0,$ consider the following commutative diagram
\[
\xymatrixcolsep{3.5pc}
\xymatrix{
	\Gamma_n\ar[d]_{\pi_1^n}\ar[dr]^{\pi_2^n}\\
	 Y \ar@{-->}[r]_{f^n}	& Y
}
\]
where $\pi_1^n$ is birational and it is an isomorphism above $X.$
There are $H_1,\dots, H_{d_X-d_V}\in |H|$ such that the intersection of $H_1,\dots, H_{d_X-d_V}$ is proper,
$V$ is an irreducible component of $\cap_{i=1}^{d_X-d_V}H_i$ and $V$ is the unique irreducible component meeting $f^{-n}(x)\cap V.$
Take $H_1',\dots, H_{d_V}'$ general in those elements of $|H|$ containing $x.$
Then the intersection of $H_1',\dots, H_{d_V}'$ and $f(V)$ at $x$ is proper. 
Since $f$ is finite, the intersection of $f^{*}(H_1'),\dots, f^*H_{d_V}'$ and $V$ is proper at every $y\in f^{-n}(x)\cap V.$

We have 
$$(\pi_1^n)^{-1}(f^{-n}(x)\cap V)\subseteq (\cap_{i=1}^{d_X-d_V}(\pi_1^n)^*H_i)\cap(\cap_{i=1}^{d_V}(\pi_2^n)^*H_i'),$$
and every point $y\in (\pi_1^n)^{-1}(f^{-n}(x)\cap V)$ is isolated in $(\cap_{i=1}^{d_X-d_V}(\pi_1^n)^*H_i)\cap(\cap_{i=1}^{d_V}(\pi_2^n)^*H_i').$
By \cite[Lemma 3.3]{Jia2021}, 
$$(H^{d_X-d_V}\cdot(f^{n})^*H^{d_V})=((\pi_1^{n})^*H_1\cdot \dots \cdot (\pi_1^{n})^*H_{d_X-d_V}\cdot (\pi_2^{n})^*H_1'\cdot \dots \cdot (\pi_2^{n})^*H_{d_V}')$$
$$\geq \#(\pi_1^n)^{-1}(f^{-n}(x)\cap V)=\#(f^{-n}(x)\cap V).$$
Then we get
$$\limsup_{n\to \infty}\#(f^{-n}(x)\cap V)^{1/n}\leq \lim_{n\to \infty}(H^{d_X-d_V}\cdot(f^{n})^*H^{d_V})^{1/n}=\la_{d_V}.$$

\medskip

Now we only need to show $$\lim_{n\to \infty}d_f^{-n}(f^n)^*\delta_x=\delta_{\eta}.$$
Because $\sM^1(|X|)$ is compact, only need to show that for every convergence subsequence $d_f^{-n_i}(f^{n_i})^*\delta_{x}, i\geq 0$,
$\lim_{i\to \infty}d_f^{-n_i}(f^{n_i})^*\delta_{x}=\delta_{\eta}.$
Set $\mu:=\lim_{i\to \infty}d_f^{-n_i}(f^{n_i})^*\delta_{x}.$
By Theorem \ref{thmRadon}, we may write 
$$\mu=\sum_{i\geq  0}^ma_i\delta_{x_i}$$ where $m\in \Z_{\geq 0}\cup \{\infty\}$, $x_i$ are distinct points, $a_i> 0$ and $\sum_{i\geq 0}a_i=1.$
Assume that $\mu\neq \delta_{\eta}$. Then we may assume that $a_0>0$ and $x_0\neq \eta.$
Set $r:=\overline{\{x_0\}}<d_X.$

Then $$\int 1_{U_{x_0}}\mu\geq \int 1_{U_{x_0}}a_0\delta_{x_0}=a_0.$$
Pick $c\in (0,a_0).$ Then there is $N\geq 0$ such that for every $i\geq N$, 
$$\frac{\sum_{y\in f^{-n_i}(x)\cap \overline{\{x_0\}}}m_{f^{n_i}}(y)}{d_f^{n_i}}=\int 1_{U_{x_0}}d_f^{-n_i}(f^{n_i})^*\delta_{x}\geq c.$$
So $\sum_{y\in f^{-n_i}(x)\cap \overline{\{x_0\}}}m_{f^{n_i}}(y)\geq cd_f^{n_i},$ then 
$\#(f^{-n_i}(x)\cap \overline{\{x_0\}})\geq S_c^{n_i}.$
By Theorem \ref{thmcountprestr} and Inequality \ref{equintvnum},  we get 
$$d_f>\la_{r}\geq \limsup_{i\to \infty}(\#(f^{-n_i}(x)\cap \overline{\{x_0\}}))^{1/n_i}\geq \liminf_{i\to \infty}(S_c^{n_i})^{1/n_i}=d_f,$$
which is a contradiction.
\endproof
\subsection{Berkovich spaces}\label{subsecberko}
 In this section, $\bk$ is a complete nonarchimedean valued field with norm $|\cdot|$.  
 See \cite{Berkovich1990} and \cite{Berkovich1993} for basic theory of Berkovich spaces.

 Let $X$ be a variety over $\bk.$
Recall that, as a topological space, Berkovich's analytification of $X$ is 
$$X^{\an}:=\{(x,|\cdot|_x)|\,\, x\in X, |\cdot|_x \text{ is a norm on }\kappa(x) \text{ which extends } |\cdot| \text{ on }\bk\},$$
endowed with the weakest topology such that
\begin{points}
\item $\tau: X^{\an}\to X$ by $(x,|\cdot|_x)\mapsto x$ is continuous;
\item for every Zariski open $U\subseteq X$ and $\phi\in O(U)$,
the map $|\phi|:\tau^{-1}(U)\to [0+\infty)$ sending $(x,|\cdot|_x)$ to $|\phi|_x$ is continuous.
\end{points}

Let $\sM(X^{\an})$ be the space of Radon measures on $X^{\an}$ and let $\sM^1(X^{\an})$ be the space of probability Radon measures on $X^{\an}.$

\subsection{Trivial norm case}
Assume that $|\cdot|$ is the trivial norm.

For every $x\in X$, let $|\cdot|_{x,0}$ be the trivial norm on $\kappa(x).$
Then we have an embedding
$\sigma: X\to X^{\an}$ sending 
$x\in X$ to $(x,|\cdot|_{x,0}).$
We have $\tau\circ\sigma=\id.$
One may check that the constructible topology on $X$ is exact the topology induced by the topology on $X^{\an}$ and the embedding $\sigma.$
Because $|X|$ is compact, $\sigma(X)$ is closed in $X^{\an}$ and $\sigma: |X|\to \sigma(|X|)$ is a homeomorphism.
\begin{rem}We note that, if $X$ is endowed with the constructible topology,  $\tau: X^{\an}\to |X|$ is no longer continuous.
\end{rem}

Using the embedding $\sigma$, Corollary \ref{corgenericseqence} can be translated to a statement on $X^{\an}.$
\begin{cor}[=Corollary \ref{corgenericseqence}]A sequence $x_n\in X, n\geq 0$ is generic if and only if in $\sM(X^{\an})$
$$\lim_{n\to \infty}\delta_{\sigma(x_n)}=\delta_{\sigma(\eta)}.$$ 
\end{cor}

\medskip

Let $f: X\to X$ be a finite flat morphism. It induces a morphism $f^{\an}: X^{\an}\to X^{\an}.$
We have $$f^{\an}\circ \sigma=\sigma\circ f \text{ and } \tau\circ f^{\an}=f\circ \tau.$$
According to \cite[Lemma 6.7]{Gignac2014a}, there is a natural pullback ${f^{\an}}^*: \sM(X^{\an})\to \sM(X^{\an}).$ 
One may check that the following diagram is commutative.

\[
\xymatrixcolsep{3.5pc}
\xymatrix{
	\sM(|X|)\ar[d]^{\sigma_*}\ar[r]^{f^*}&\sM(|X|)\ar[d]^{\sigma_*}\\
	 \sM(X^{\an}) \ar[r]_{{f^{\an}}^*}	& \sM(X^{\an})
}
\]

Then we may translate Theorem \ref{thmseqd} to a statement on $X^{\an}.$
\begin{thm}[=Theorem \ref{thmseqd}]
Let $f: X\to X$ be a flat and finite endomorphism of a quasi-projective variety. Assume that 
\begin{equation}\label{equladdom}d_f:=\la_{\dim X}(f)>\max_{1\leq i\leq \dim X-1} \la_i.
\end{equation}
If the field extension $\bk(X)/f^*\bk(X)$ is separable, then for every $x\in X(\bk)$ with $\overline{\cup_{i\geq 0}f^{-i}(x)}=X,$
 $$\lim_{n\to \infty}d_f^{-n}(f^n)^*\delta_{\sigma(x)}=\delta_{\sigma(\eta)}.$$
\end{thm}

\subsection{Reduction}\label{subsectionreduction}
Let $\bk^{\circ}$ be the valuation ring of $\bk$ and $\bk^{\circ\circ}$ the maximal ideal of $\bk^{\circ}.$
Set $\widetilde{\bk}:=\bk^{\circ}/\bk^{\circ\circ}$ the residue field of $\bk.$ 
Let $\sX$ be a flat projective scheme over $\bk^{\circ}.$ Denote by $X_0$ its special fiber, it is a (maybe reducible) variety over $\widetilde{\bk}.$
Let $X$ be the generic fiber of $\sX.$ Let $Y_1,\dots, Y_m$ be the irreducible components of $X_0$ and $\eta_i,i=1,\dots,m$ the generic points of $Y_i.$
Set $\xi_i$ the unique point in $\red^{-1}(\eta_i).$

\medskip

Denote by $\red: X^{\an}\to X_0$ the reduction map. It is anti-continuous i.e. for every Zariski open subset $U$ of $X_0$,  $\red^{-1}(U)$ is closed.
In particular, for constructible topology on $X_0$, $\red: X^{\an}\to |X_0|$ Borel measurable. 

\medskip

For every $\mu\in \sM(X^{\an})$,  we may define its push forward $\red_*\mu\in \sM(|X_0|)$ as follows: 
For every $\phi\in C(|X_0|)$, $$\int \phi\, \red_*\mu:=\int (\red^*\phi)\, \mu.$$
Because $\red^*\phi$ is Borel measurable and bounded, $\int (\red^*\phi) \mu$ is well defined and we have $|\int (\red^*\phi) \mu|\leq \|\phi\|_{\infty}\mu(X^{\an}).$
We note that, in general, $\red_*: \sM(X^{\an})\to \sM(|X_0|)$  is not continuous. 

\begin{exe}\label{exerednotmeacon}
Let $\sX=\P^N_{\bk^{\circ}}.$ 
Let $x_n, n\geq 0$ be the Gauss point of the polydisc $\{|T_i|\leq 1-1/(n+2), i=1,\dots, N\}\subseteq (\A^N)^{\an}\subseteq (\P^N)^{\an}.$
We have $\delta_{x_n}\to \xi_1$ as $n\to \infty$, but for every $n\geq 0$, 
$$\red_*\delta_{x_n}=\delta_{\red(x_n)}=\delta_{[1:0:\dots:0]}\neq \delta_{\eta_1}=\red_*\delta_{\xi_1}.$$
\end{exe}

\begin{pro}\label{prolimgen}Let $\mu_n\in \sM^1(X^{\an}), n\geq 0$ be a sequence of probability Radon measures on $X^{\an}.$
Assume that there are $a_i\geq 0, i=1,\dots,m$ with $\sum_{i=1}^ma_i=1$ such that 
$$\red_*(\mu_n)\to \sum_{i=1}^m a_i\delta_{\eta_i}$$
as $n\to \infty.$
Then we have $$\mu_n\to \sum_{i=1}^m a_i\delta_{\xi_i}$$
as $n\to \infty.$
\end{pro}
\proof
Because $X^{\an}$ is compact, $\sM^1(X^{\an})$ is weak-$\ast$ compact. So we may assume that 
$$\lim_{n\to \infty}\mu_n=\mu$$ for some $\mu\in \sM^1(X^{\an}).$
We first show that $\Supp\, \mu\subseteq \{\xi_1\dots,\xi_m\}.$
Otherwise $\mu(X^{\an}\setminus \{\xi_1\dots,\xi_m\})=1.$ Then there is a compact subset $K$ of $X^{\an}\setminus \{\xi_1\dots,\xi_m\}$ such that $\mu(K)>0.$
For every $x\in K$, set $V_x:=\red^{-1}(\overline{\red(x)}).$ It is an open neighborhood of $x$ in $X^{\an}\setminus \{\xi_1\dots,\xi_m\}.$
Because $K$ is compact, there is one $x\in K$ such that $\mu(V_x)>0.$ 
Set $Z:=\overline{\red(x)}.$  There is a compact subset $S\subseteq V_x$ such that $\mu(S)>0.$
By Urysohn's Lemma, there is a continuous function $\chi: X^{\an}\to [0,1]$ such that $\chi|_S=1$ and $\chi|_{X^{\an}\setminus V_x}=0.$
Then we have 
$$0=\lim_{n\to \infty}\int 1_Z\,\red_*\mu_n=\lim_{n\to \infty}\int (\red^*1_Z)\,\mu_n=\lim_{n\to \infty}\int 1_{V_x}\,\mu_n$$
$$\geq \lim_{n\to \infty}\int \chi\,\mu_n=\int \chi \mu\geq \mu(S)>0,$$
which is a contradiction.

Now  we may write $\mu=\sum_{i=1}^m b_i\delta_{\xi_i}$ with $b_i\geq 0$ and $\sum_{i=1}^m b_i=1.$
For each $i=1,\dots,m$,  set $U_i:=Z_i\setminus (\cup_{j\neq i}Z_j).$
Then $\red^{-1}(U_i)$ is a closed subset contained in the open subset $\red^{-1}(Z_i).$
By Urysohn's Lemma, there is a continuous function $\chi_i: X^{\an}\to [0,1]$ such that $\chi|_{\red^{-1}(U_i)}=1$ and $\chi|_{X^{\an}\setminus \red^{-1}(Z_i)}=0.$
Then we have 
$$b_i=\int\chi_i\mu=\lim_{n\to \infty}\int \chi_i\mu_n$$
$$\geq \lim_{n\to \infty}\mu_n(\red^{-1}(U_i))=\lim_{n\to \infty}\int 1_{U_i}\,\red_*\mu_n$$
$$=\int 1_{U_i}\,(\sum_{j=1}^m a_j\delta_{\eta_j})=a_i.$$
Because $\sum_{i=1}^{m}b_i=\sum_{i=1}^m a_i=1$, we get $b_i=a_i$ for every $i=1,\dots,m.$
This concludes the proof.
\endproof

Now assume that $X_0$ is irreducible and smooth. Denote by $\eta$ the generic point of $X_0$ and $\xi$ the unique point in $\red^{-1}(\eta).$
Let $F:\sX\to \sX$ be a finite endomorphism. Denote by $f,f_0$ the restriction of $F$ to $X,X_0.$
We note that for $i=0,\dots, \dim X$, one has $\la_i(f)=\la_i(f_0).$

\medskip

By Theorem \ref{thmseqd} and Proposition \ref{prolimgen}, we get the following equidistribution result for endomorphisms of good reductions.
\begin{cor}\label{corseqdber}
Assume that $$d_f:=\la_{\dim X}(f)>\max_{1\leq i\leq \dim X-1} \la_i.$$
If the field extension $\widetilde{\bk}(X_0)/f_0^*\widetilde{\bk}(X_0)$ is separable, then for every $x\in X(\bk)$ with $\overline{\cup_{i\geq 0}f_0^{-i}(\red{x})}=X_0,$
 $$\lim_{n\to \infty}d_f^{-n}(f^n)^*\delta_x=\delta_{\xi}.$$
\end{cor}
One may compare Corollary \ref{corseqdber} with \cite[Theorem A]{Gignac2014a} for  polarized endomorphism. See \cite{Guedj2005,Dinh2015} for according result for complex topology.

\newpage

\bibliography{dd}

\begin{thebibliography}{10}

\bibitem{Amerik}
Ekaterina Amerik.
\newblock Existence of non-preperiodic algebraic points for a rational self-map
  of infinite order.
\newblock {\em Math. Res. Lett.}, 18(2):251--256, 2011.

\bibitem{E.Amerik2011}
Ekaterina Amerik, Fedor Bogomolov, and Marat Rovinsky.
\newblock Remarks on endomorphisms and rational points.
\newblock {\em Compositio Math.}, 147:1819--1842, 2011.

\bibitem{Amerik2008}
Ekaterina Amerik and Fr{\'e}d{\'e}ric Campana.
\newblock Fibrations m\'eromorphes sur certaines vari\'et\'es \`a fibr\'e
  canonique trivial.
\newblock {\em Pure Appl. Math. Q.}, 4(2, part 1):509--545, 2008.

\bibitem{Bell2017}
Jason~P. Bell, Dragos Ghioca, and Zinovy Reichstein.
\newblock On a dynamical version of a theorem of {R}osenlicht.
\newblock {\em Ann. Sc. Norm. Super. Pisa Cl. Sci. (5)}, 17(1):187--204, 2017.

\bibitem{Bell2017a}
Jason~P. Bell, Dragos Ghioca, Zinovy Reichstein, and Matthew Satriano.
\newblock On the {M}edvedev-{S}canlon conjecture for minimal threefolds of
  nonnegative {K}odaira dimension.
\newblock {\em New York J. Math.}, 23:1185--1203, 2017.

\bibitem{Bell2010}
Jason~P. Bell, Dragos Ghioca, and Thomas~J. Tucker.
\newblock The dynamical {M}ordell-{L}ang problem for \'etale maps.
\newblock {\em Amer. J. Math.}, 132(6):1655--1675, 2010.

\bibitem{Bell}
Jason~P. Bell, Dragos Ghioca, and Thomas~J. Tucker.
\newblock Applications of {$p$}-adic analysis for bounding periods for
  subvarieties under \'etale maps.
\newblock {\em Int. Math. Res. Not. IMRN}, (11):3576--3597, 2015.

\bibitem{Bell2015}
Jason~P. Bell, Dragos Ghioca, and Thomas~J. Tucker.
\newblock The dynamical {M}ordell-{L}ang problem for {N}oetherian spaces.
\newblock {\em Funct. Approx. Comment. Math.}, 53(2):313--328, 2015.

\bibitem{Bell2016}
Jason~P. Bell, Dragos Ghioca, and Thomas~J. Tucker.
\newblock {\em The dynamical {M}ordell-{L}ang conjecture}, volume 210 of {\em
  Mathematical Surveys and Monographs}.
\newblock American Mathematical Society, Providence, RI, 2016.

\bibitem{Bell2020}
Jason~P. Bell, Fei Hu, and Matthew Satriano.
\newblock Height gap conjectures, {$D$}-finiteness, and a weak dynamical
  {M}ordell-{L}ang conjecture.
\newblock {\em Math. Ann.}, 378(3-4):971--992, 2020.

\bibitem{Benoist2011}
Olivier Benoist.
\newblock Le th\'{e}or\`eme de {B}ertini en famille.
\newblock {\em Bull. Soc. Math. France}, 139(4):555--569, 2011.

\bibitem{Berkovich1990}
Vladimir~G. Berkovich.
\newblock {\em Spectral theory and analytic geometry over non-Archimedean
  fields.}
\newblock Number~33 in Mathematical Surveys and Monographs. American
  Mathematical Society, 1990.

\bibitem{Berkovich1993}
Vladimir~G. Berkovich.
\newblock \'{E}tale cohomology for non-{A}rchimedean analytic spaces.
\newblock {\em Inst. Hautes \'{E}tudes Sci. Publ. Math.}, (78):5--161 (1994),
  1993.

\bibitem{Invarianthypersurfaces}
Serge Cantat.
\newblock Invariant hypersurfaces in holomorphic dynamics.
\newblock {\em Math. Research Letters}, 17(5):833--841, 2010.

\bibitem{Cantat2019a}
Serge Cantat, Andriy Regeta, and Junyi Xie.
\newblock Families of commuting automorphisms, and a characterization of the
  affine space.
\newblock arXiv:1912.01567, 2019.

\bibitem{Corvaja2021}
Pietro Corvaja, Dragos Ghioca, Thomas Scanlon, and Umberto Zannier.
\newblock The dynamical {M}ordell-{L}ang conjecture for endomorphisms of
  semiabelian varieties defined over fields of positive characteristic.
\newblock {\em J. Inst. Math. Jussieu}, 20(2):669--698, 2021.

\bibitem{Dang}
Nguyen-Bac Dang.
\newblock {D}egrees of {I}terates of {R}ational {M}aps on {N}ormal {P}rojective
  {V}arieties.

\bibitem{Dang2020}
Nguyen-Bac Dang.
\newblock Degrees of iterates of rational maps on normal projective varieties.
\newblock {\em Proc. Lond. Math. Soc. (3)}, 121(5):1268--1310, 2020.

\bibitem{debarre}
Olivier Debarre.
\newblock {\em Higher-dimensional algebraic geometry}.
\newblock Universitext. Springer-Verlag, New York, 2001.

\bibitem{Dinh2015}
Tien-Cuong Dinh, Vi\^{e}t-Anh Nguy\^{e}n, and Tuyen~Trung Truong.
\newblock Equidistribution for meromorphic maps with dominant topological
  degree.
\newblock {\em Indiana Univ. Math. J.}, 64(6):1805--1828, 2015.

\bibitem{Dinh2005}
Tien-Cuong Dinh and Nessim Sibony.
\newblock Une borne sup\'erieure pour l'entropie topologique d'une application
  rationnelle.
\newblock {\em Ann. of Math. (2)}, 161(3):1637--1644, 2005.

\bibitem{fa}
Najmuddin Fakhruddin.
\newblock Questions on self maps of algebraic varieties.
\newblock {\em J. Ramanujan Math. Soc.}, 18(2):109--122, 2003.

\bibitem{Fakhruddin2014}
Najmuddin Fakhruddin.
\newblock The algebraic dynamics of generic endomorphisms of {$\Bbb{P}^n$}.
\newblock {\em Algebra Number Theory}, 8(3):587--608, 2014.

\bibitem{Favre2000a}
Charles Favre.
\newblock {\em Dynamique des applications rationnelles}.
\newblock PhD thesis, 2000.

\bibitem{Ghioca2019a}
Dragos Ghioca.
\newblock The dynamical {M}ordell-{L}ang conjecture in positive characteristic.
\newblock {\em Trans. Amer. Math. Soc.}, 371(2):1151--1167, 2019.

\bibitem{Ghioca2018b}
Dragos Ghioca and Fei Hu.
\newblock Density of orbits of endomorphisms of commutative linear algebraic
  groups.
\newblock {\em New York J. Math.}, 24:375--388, 2018.

\bibitem{Ghioca}
Dragos Ghioca and Sina Saleh.
\newblock Zariski dense orbits for regular self-maps of tori in positive
  characteristic. (with sina saleh).

\bibitem{Ghioca2019}
Dragos Ghioca and Matthew Satriano.
\newblock Density of orbits of dominant regular self-maps of semiabelian
  varieties.
\newblock {\em Trans. Amer. Math. Soc.}, 371(9):6341--6358, 2019.

\bibitem{Ghioca2017a}
Dragos Ghioca and Thomas Scanlon.
\newblock Density of orbits of endomorphisms of abelian varieties.
\newblock {\em Trans. Amer. Math. Soc.}, 369(1):447--466, 2017.

\bibitem{Ghioca2018}
Dragos Ghioca and Junyi Xie.
\newblock The dynamical mordell¨clang conjecture for skew-linear self-maps.
  appendix by michael wibmer.
\newblock {\em International Mathematics Research Notices}, page rny211, 2018.

\bibitem{Gignac2014a}
William Gignac.
\newblock Equidistribution of preimages over nonarchimedean fields for maps of
  good reduction.
\newblock {\em Ann. Inst. Fourier (Grenoble)}, 64(4):1737--1779, 2014.

\bibitem{Gignac2014}
William Gignac.
\newblock Measures and dynamics on {N}oetherian spaces.
\newblock {\em J. Geom. Anal.}, 24(4):1770--1793, 2014.

\bibitem{Gizatullin1980}
M.~H. Gizatullin.
\newblock Rational {$G$}-surfaces.
\newblock {\em Izv. Akad. Nauk SSSR Ser. Mat.}, 44(1):110--144, 239, 1980.

\bibitem{EGA-IV-I}
A.~Grothendieck.
\newblock \'{E}l\'{e}ments de g\'{e}om\'{e}trie alg\'{e}brique. {IV}. \'{E}tude
  locale des sch\'{e}mas et des morphismes de sch\'{e}mas. {I}.
\newblock {\em Inst. Hautes \'{E}tudes Sci. Publ. Math.}, (20):259, 1964.

\bibitem{Guedj2005}
Vincent Guedj.
\newblock Ergodic properties of rational mappings with large topological
  degree.
\newblock {\em Ann. of Math. (2)}, 161(3):1589--1607, 2005.

\bibitem{Jia2021}
Jia Jia, Takahiro Shibata, Junyi Xie, and De-Qi Zhang.
\newblock {E}ndomorphisms of quasi-projective varieties -- towards {Z}ariski
  dense orbit and {K}awaguchi-{S}ilverman conjectures.
\newblock arXiv:2104.05339, 2021.

\bibitem{Jia2020}
Jia Jia, Junyi Xie, and De-Qi Zhang.
\newblock Surjective endomorphisms of projective surfaces -- the existence of
  infinitely many dense orbits.
\newblock arXiv:2005.03628, 2020.

\bibitem{Kawaguchi2016}
Shu Kawaguchi and Joseph~H. Silverman.
\newblock On the dynamical and arithmetic degrees of rational self-maps of
  algebraic varieties.
\newblock {\em J. Reine Angew. Math.}, 713:21--48, 2016.

\bibitem{Kawaguchi2020}
Shu Kawaguchi and Joseph~H. Silverman.
\newblock Erratum to: ``{O}n the dynamical and arithmetic degrees of rational
  self-maps of algebraic varieties" ({J}. {R}eine {A}ngew. {M}ath. 713 (2016),
  21--48).
\newblock {\em J. Reine Angew. Math.}, 761:291--292, 2020.

\bibitem{Lesieutre2021}
John Lesieutre and Matthew Satriano.
\newblock Canonical {H}eights on {H}yper-{K}\"{a}hler {V}arieties and the
  {K}awaguchi--{S}ilverman {C}onjecture.
\newblock {\em Int. Math. Res. Not. IMRN}, (10):7677--7714, 2021.

\bibitem{Lipman1978}
Joseph Lipman.
\newblock Desingularization of two-dimensional schemes.
\newblock {\em Ann. of Math. (2)}, 107(1):151--207, 1978.

\bibitem{Matsuzawa2020a}
Yohsuke Matsuzawa.
\newblock On upper bounds of arithmetic degrees.
\newblock {\em Amer. J. Math.}, 142(6):1797--1820, 2020.

\bibitem{Matsuzawa2018}
Yohsuke Matsuzawa, Kaoru Sano, and Takahiro Shibata.
\newblock Arithmetic degrees and dynamical degrees of endomorphisms on
  surfaces.
\newblock {\em Algebra Number Theory}, 12(7):1635--1657, 2018.

\bibitem{Medvdevv1}
Alice Medvedev and Thomas Scanlon.
\newblock Polynomial dynamics.
\newblock arXiv:0901.2352v1.

\bibitem{Medvdev}
Alice Medvedev and Thomas Scanlon.
\newblock Invariant varieties for polynomial dynamical systems.
\newblock {\em Ann. of Math. (2)}, 179(1):81--177, 2014.

\bibitem{Petsche2015}
Clayton Petsche.
\newblock On the distribution of orbits in affine varieties.
\newblock {\em Ergodic Theory Dynam. Systems}, 35(7):2231--2241, 2015.

\bibitem{Poonen2014}
Bjorn Poonen.
\newblock {$p$}-adic interpolation of iterates.
\newblock {\em Bull. Lond. Math. Soc.}, 46(3):525--527, 2014.

\bibitem{Russakovskii1997}
Alexander Russakovskii and Bernard Shiffman.
\newblock Value distribution for sequences of rational mappings and complex
  dynamics.
\newblock {\em Indiana Univ. Math. J.}, 46(3):897--932, 1997.

\bibitem{Truong2020}
Tuyen~Trung Truong.
\newblock Relative dynamical degrees of correspondences over a field of
  arbitrary characteristic.
\newblock {\em J. Reine Angew. Math.}, 758:139--182, 2020.

\bibitem{Xie2014}
Junyi Xie.
\newblock Dynamical {M}ordell-{L}ang conjecture for birational polynomial
  morphisms on {$\Bbb A^2$}.
\newblock {\em Math. Ann.}, 360(1-2):457--480, 2014.

\bibitem{Xie2015}
Junyi Xie.
\newblock Periodic points of birational transformations on projective surfaces.
\newblock {\em Duke Math. J.}, 164(5):903--932, 2015.

\bibitem{Xie2017a}
Junyi Xie.
\newblock The dynamical {M}ordell-{L}ang conjecture for polynomial
  endomorphisms of the affine plane.
\newblock {\em Ast\'{e}risque}, (394):vi+110, 2017.

\bibitem{Xie2017}
Junyi Xie.
\newblock The existence of {Z}ariski dense orbits for polynomial endomorphisms
  of the affine plane.
\newblock {\em Compos. Math.}, 153(8):1658--1672, 2017.

\bibitem{Xie2018}
Junyi Xie.
\newblock Algebraic dynamics of the lifts of {F}robenius.
\newblock {\em Algebra Number Theory}, 12(7):1715--1748, 2018.

\bibitem{Xie2019}
Junyi Xie.
\newblock The existence of {Z}ariski dense orbits for endomorphisms of
  projective surfaces (with an appendix in collaboration with {T}homas
  {T}ucker).
\newblock arXiv:1905.07021, 2019.

\bibitem{zhang}
Shou-Wu Zhang.
\newblock Distributions in algebraic dynamics.
\newblock In {\em Surveys in differential geometry. {V}ol. {X}}, volume~10 of
  {\em Surv. Differ. Geom.}, pages 381--430. Int. Press, Somerville, MA, 2006.

\end{thebibliography}
\end{document}